\documentclass{article}
\usepackage{fullpage}
\usepackage[utf8]{inputenc}
\usepackage{subfigure}
\usepackage[round]{natbib}
\usepackage{amsfonts,amssymb,amsmath,dsfont,amsthm}
\usepackage{xspace}
\usepackage{graphicx}
\usepackage[colorlinks=true,citecolor=blue]{hyperref}
\usepackage{bm}
\usepackage{comment}
\usepackage[capitalize,nameinlink,noabbrev]{cleveref} 

\newtheorem{theorem}{Theorem}
\newtheorem{proposition}[theorem]{Proposition}
\newtheorem{lemma}[theorem]{Lemma}
\newtheorem{corollary}[theorem]{Corollary}

\newtheorem*{theorem*}{Theorem}
\newtheorem*{proposition*}{Proposition}
\newtheorem*{lemma*}{Lemma}
\newtheorem*{corollary*}{Corollary}
\newtheorem*{definition*}{Definition}

\newcommand{\bigO}{\mathcal{O}}
\newcommand{\exactbigO}{\Theta}
\newcommand{\smallo}{o}
\newcommand{\integers}{\mathds{Z}}
\newcommand{\reals}{\mathds{R}}

\newcommand{\indic}{\mathds{1}}

\newcommand{\vect}[1]{\bm{#1}}

\newcommand*{\eg}{\textit{e.g.}\@\xspace}

\newcommand{\Real}{\operatorname{Re}}
\renewcommand{\Im}{\operatorname{Im}}
\newcommand{\exphad}{\odot}

\newcommand{\sg}{\operatorname{SG}}
\newcommand{\sgk}{\sg^{(k)}}
\newcommand{\csg}{\operatorname{CSG}}
\newcommand{\csgk}{\csg^{(k)}}
\newcommand{\csgkn}{\csgk_n}
\newcommand{\tsg}{\tilde{\sg}}
\newcommand{\tsgk}{\tilde{\sg}^{(k)}}
\newcommand{\tcsg}{\tilde{\csg}}
\newcommand{\tcsgk}{\tilde{\csg}^{(k)}}

\newcommand{\tA}{\tilde{A}}
\newcommand{\tB}{\tilde{B}}
\newcommand{\tF}{\tilde{F}}
\newcommand{\tS}{\tilde{S}}

\newcommand{\vt}{\vect{t}}

\newcommand{\vx}{\vect{x}}

\newcommand{\vz}{\vect{z}}

\newcommand{\vone}{\vect{1}}
\newcommand{\valpha}{\vect{\alpha}}

\newcommand{\sgkn}{\sg^{(k)}_n}

\begin{document}
\title{Asymptotic expansion of regular and connected regular graphs}
\author{\'Elie de Panafieu}
\date{Nokia Bell Labs, France\\[0.2cm] \today}
\maketitle

\begin{abstract}
We derive the asymptotic expansion (asymptotics with an arbitrary number of error terms) of $k$-regular graphs by applying the Laplace method on a recent exact formula from Caizergues and de Panafieu (2023). We also deduce the asymptotic expansion of connected $k$-regular graphs using standard techniques for divergent series developed by Wright (1970) and Bender (1975), and quantify its closeness to the asymptotic expansion of $k$-regular graphs.
\end{abstract}

        \section{Introduction}
        \label{sec:introduction}

A graph $G$ is a pair $(V, E)$,
where $V = \{1, \ldots, n(G)\}$
denotes the set of labeled vertices,
and $E$ is the set of unlabeled unoriented edges.
Loops and multiple edges are forbidden.
The \emph{degree} of a vertex is
the number of edges containing it.
Graphs where all vertices have the same degree $k$  
are called \emph{$k$-regular}.
For example $0$-regular graphs are
graphs without any edge,
$1$-regular graphs are sets of
pairs of vertices linked by an edge,
$2$-regular graphs are sets of cycles of length at least $3$,
because cycles of length $1$ are loops,
and cycles of length $2$ are double edges.

Given a formal power series $\tF(z)$,
let us write
\[
    f_n \approx a_n \tF(n^{-1})
\]
if for any $r \geq 0$, we have
\[
    f_n = a_n
    \bigg(
    \sum_{\ell=0}^{r-1} [z^\ell] \tF(z) n^{-\ell}
    + \bigO(n^{-r})
    \bigg).
\]
We then say that $a_n \tF(n^{-1})$
is the \emph{asymptotic expansion} of $f_n$.
Our main result, \cref{th:main:result},
is the asymptotic expansion
for the number $\sgkn$ of $k$-regular on $n$ vertices
for $k \geq 2$
\[
    \sgkn \approx
    \frac{(n k / e)^{n k / 2}}{k!^n}
    \frac{e^{-(k^2-1)/4}}{\sqrt{2}}
    \tsgk(n^{-1})
\]
for a formal power series $\tsgk(z)$
with $\tsgk(0) = 2$,
whose coefficients are functions of $k$,
explicitly given in the theorem.
This extends several previous results.
\cite{read1959some} obtained this asymptotic expansion for $k=3$.
For a general fixed $k$, \cite{bender_asymptotic_1978} and \cite{Bo80}
derived the main asymptotics of $\sgkn$,
and \cite{mckay1983applications} and \cite{mckay1991asymptotic}
conjectured the first three terms of the asymptotic expansion.
A linear recurrence with polynomial coefficients for $\sgkn$
has been obtained by \cite{read1960enumeration} and \cite{chen1999enumeration} for $k=3$,
by \cite{read1980number} and \cite{goulden_hammond_1983} for $k=4$
and up to $k=7$ by \cite{chyzak2024differential}.
From those recurrences, an asymptotic expansion
is likely computable (\cite{kauers_mathematica_2011}, \cite{zeilberger_asyrec_2008})
but requires some additional work.

We also investigate the number $\csgkn$
of connected $k$-regular graphs on $n$ vertices,
proving in \cref{th:csgkn}
\[
    \csgkn \approx
    \frac{(n k / e)^{n k / 2}}{k!^n}
    \frac{e^{-(k^2-1)/4}}{\sqrt{2}}
    \tcsgk(n^{-1})
\]
for any $k \geq 3$.
This extends the result $\csgkn \sim \sgkn$
from \cite{Bollobas} and \cite{wormald1981asymptotic}
(main asymptotics).
Going further, we prove
for any $k \geq 3$
\[
    \csgkn =
    \sgkn
    \left(1 + \exactbigO(n^{-(k+1)(k-2)/2}) \right)
\]
in \cref{th:link:asympt}.

Our work relies on a recent exact formula
for $k$-regular graphs
from \cite{caizergues2023exact},
to which the Laplace method is applied
(a classic technique for extracting the asymptotics of parametric integrals, see \eg \cite{bruijn58}, \cite{FS09} \cite{erdelyi1956asymptotic}, \cite{Olver74}, \cite{wong2001asymptotic,PW13}).
The asymptotic expansion of
connected $k$-regular graphs,
stated in \cref{th:csgkn},
then follows by application of techniques
for extracting the asymptotics of the coefficients
of divergent series,
due to \cite{Wright70} and \cite{Be75}.

After introducing the problem,
presenting related work
and setting some definitions
in \cref{sec:introduction},
\cref{sec:regular:graphs} presents our main result
on the asymptotic expansion of regular graphs,
then \cref{sec:connected:regular:graphs}
extends this result to connected regular graphs.
To simplify the reading of the paper,
we moved all the proofs at the end,
in \cref{sec:proofs}.

    \subsection{Related work}

The enumeration of regular graphs
started in 1959 when Read
obtained linear recurrences with polynomial coefficients
for the numbers of $3$-regular graphs
and connected $3$-regular graphs (\cite{read1959some}, \cite{read1959enumeration}, \cite{read1960enumeration}, \cite{read1970some}),
and showed how to extract asymptotic expansions from them.
Since then, those results have been extended
in various directions,
including graphs with a given degree sequence,
graphs with degrees restricted to a given set,
and $k$-regular graphs where $k$ varies
with the number of vertices
(see \eg \cite{mckay1991asymptotic}).
This area of research is too rich
to be fairly covered in this paper,
so we refer the interested reader to the surveys
by \cite{wormald2018asymptotic} and \cite{Wo99},
and restrict our review of the literature
to the exact and asymptotic enumeration
of $k$-regular and connected $k$-regular graphs
for fixed $k$.

Let us start with results on the number $\sgkn$
of $k$-regular graphs on $n$ vertices.
\cite{read1959enumeration} and \cite{read1959some}
provided a general exact expression
for $\sgkn$, that was not directly amenable
to asymptotic analysis.
Using a probabilistic approach and introducing what is
known today as the \emph{configuration model},
\cite{bender_asymptotic_1978} and \cite{Bo80}
derived the main asymptotics of $\sgkn$ for any fixed $k$.
A recurrence for $\sg^{(4)}_n$
was obtained by \cite{read1980number},
and a different recurrence was derived by
\cite{goulden_hammond_1983}.
\cite{mckay1983applications} and \cite{mckay1991asymptotic}
correctly conjectured the first three terms
of the asymptotic expansion of $\sgkn$.
\cite{chen1999enumeration} found the recurrence for $\sg^{(3)}_n$ using a different approach.
The next important step was the proof by \cite{rf1990gessel}
that the generating function of $\sgkn$ is D-finite.
This is equivalent with the fact that for any $k$,
the sequence $\sgkn$ satisfies a linear recurrence
with polynomial coefficients.
The corresponding differential equations were calculated by
\cite{CMS05} and \cite{mishna_automatic_2005},
and \cite{chyzak2024differential}
recently obtained the differential equation
up to $k = 7$.
As mentioned earlier, an asymptotic expansion
is likely computable from this recurrence
(\cite{kauers_mathematica_2011}, \cite{zeilberger_asyrec_2008}).
However, the completeness of the theory for the computation
of an asymptotic expansion from a linear recurrence
(\cite{birkhoff1933analytic})
is questioned (\cite[Section VIII.7 p.~581]{FS09}, \cite{wong1992asymptotic}).
Finally, \cite{caizergues2023exact} provided
a new exact expression for $\sgkn$.
The present article extracts
an asymptotic expansion from it.

We now turn to connected $k$-regular graphs.
The exact and asymptotic enumeration of connected $3$-regular graphs
was first achieved by \cite{read1960enumeration} and \cite{read1970some}
and was extended by \cite{wormald1979enumeration}
to more general notions of connectivity.
Then \cite{Bollobas} and \cite{wormald1981asymptotic} proved independently
that $k$-regular graphs are asymptotically almost surely connected.
Finally, \cite{Wright70} linked the asymptotic expansions
of a combinatorial family and sets of objects from this family.
Thus, it is no surprise that our asymptotic expansion of regular graphs
translates to an asymptotic expansion for connected regular graphs.
We chose to explicit this link in \cref{sec:regular:graphs}.

\paragraph{What does it mean to consider a formal value $k$?}
The results of \cite{rf1990gessel}, \cite{mishna_automatic_2005}
show that, for any fixed $k$
and given enough computational power,
one can compute the differential equation
characterizing the generating function $\sgk(z)$
of $k$-regular graphs,
from which an asymptotic expansion for the number $\sgkn$
of $k$-regular graphs on $n$ vertices
can often be extracted.
For example, stopping at order $1$ for the sake of simplicity
and keeping ``$k$'' non-evaluated in part of the formulas
to highlight their similarities, one finds
for $k=3$
\[
    \sg^{(3)}_n =
    \frac{(nk/e)^{nk/2}}{k!^n}
    \frac{e^{-(k^2-1)/4}}{\sqrt{2}}
    \left(
    2 - \frac{71}{18} n^{-1} + \bigO(n^{-2})
    \right)
\]
and for $k=4$
\[
    \sg^{(4)}_n =
    \frac{(nk/e)^{nk/2}}{k!^n}
    \frac{e^{-(k^2-1)/4}}{\sqrt{2}}
    \left(
    2 - \frac{235}{24} n^{-1} + \bigO(n^{-2})
    \right).
\]
In contrast, our result provides
explicit expressions for the terms
of the asymptotic expansion,
where the value $k$ appears as formal.
We will prove in particular for any $k \geq 3$
\[
    \sgkn =
    \frac{(nk/e)^{nk/2}}{k!^n}
    \frac{e^{-(k^2-1)/4}}{\sqrt{2}}
    \left(
    2 - \frac{1}{6} (k^4 - 2 k^2 + 3 k - 1) (k n)^{-1} + \bigO(n^{-2})
    \right).
\]
More terms of the asymptotic expansion
are provided in \cref{sec:numerical},
and terms corresponding to connected regular graphs
in \cref{sec:connected:numerical}.

\paragraph{Could the asymptotic expansion be derived using previous techniques?}
To our knowledge, this is the first time the asymptotic expansion
of $k$-regular graphs is obtained for a general value $k$.
Could previous methods have obtained the same result?
%
Many different techniques have been developed
to enumerate regular graphs.
In fact, we suspect regular graphs have become
a playground to test mathematical techniques
and illustrate their scope.
This makes it difficult to answer categorically
the question of this paragraph.
We can however highlight some key differences
with previous approaches.

While counting $k$-regular graphs is challenging,
enumerating some models of $k$-regular multigraphs
(where loops and multiple edges are allowed)
is simple.
This classical observation
(see \eg the configuration model \cite{bender_asymptotic_1978},
\cite{Bo80} or \cite{caizergues2023exact})
leads to counting $k$-regular graphs
by deformation from $k$-regular multigraphs,
removing loops and multiple edges.
Previous works have used inclusion-exclusion
to remove loops and multiple edges from multigraphs,
but they did not have access
to the full description of all possible ways
loops and multiple edges can overlap.
The exact formula from \cite{caizergues2023exact}
enumerating $k$-regular graphs contains this information
(using inversion instead of inclusion-exclusion),
and the present work builds upon it to extract
the asymptotic expansion.

Some previous works did not rely on a combinatorial decomposition
to enumerate graphs with constrained degrees.
For example, \cite{mckay1991asymptotic}
consider the generating function of all graphs
with $x_j$ marking the degree of vertex $j$
\[
    \prod_{1 \leq i < j} (1 + x_i x_j)
\]
and write the number of graphs
with degree sequence $(d_1, d_2, \ldots)$ on $n$ vertices
as a coefficient extraction, expressed by a Cauchy integral.
For $k$-regular graphs, this leads to
\[
    \sgkn
    =
    [x_1^d \cdots x_n^d]
    \prod_{1 \leq i < j \leq n} (1 + x_i x_j)
    =
    \frac{1}{(2 i \pi)^n}
    \oint^n
    \prod_{1 \leq i < j \leq n} (1 + z_i z_j)
    \frac{d z_1}{z_1^{d+1}}
    \cdots
    \frac{d z_n}{z_n^{d+1}}.
\]
The main difference with the present work
is that the dimension $n$ of the domain of integration
tends to infinity,
while our integrals are on a fixed dimension $k$,
simplifying the extraction of an asymptotic expansion.

    \subsection{Definitions and elementary properties}
    \label{sec:def}

\paragraph{Double factorial.}
The double factorial notation for odd numbers stands for
\[
    (2n-1)!! =
    \frac{(2n)!}{2^n n!}
    =
    (2n-1) (2n-3) (2n-5) \cdots 1.
\]
It has the following integral representation
\[
    (2n-1)!! =
    \frac{1}{\sqrt{2 \pi}}
    \int_{\reals}
    t^{2n} e^{-t^2/2} dt.
\]

\paragraph{Hadamard product.}
The \emph{exponential Hadamard product} is denoted by and defined as
\[
    \sum_n a_n \frac{z^n}{n!}
    \odot_z
    \sum_n b_n \frac{z^n}{n!}
    =
    \sum_n a_n b_n \frac{z^n}{n!}.
\]
For example, we have
\[
    F(z) \odot_z e^z = F(z).
\]
The exponential Hadamard product was introduced in
\cite[Section 2.1, p.~64]{BLL97}
as \emph{Cartesian product} or just \emph{Hadamard product}.
This notion (``\emph{cet ami oublié}'') can be
traced back to \cite[Theorem 3, Equation (8)]{Joy81}.
To avoid confusion with the more common ordinary Hadamard product,
we keep the word ``\emph{exponential}''.
We define the evaluated exponential Hadamard product as
\[
    F(z) \odot_{z = x} G(z)
    =
    \left( F(z) \odot_z G(z) \right)_{|z = x}.
\]
It satisfies the elementary property
\[
    F(z) \odot_{z = x} G(z)
    =
    F(z x) \odot_{z = 1} G(z)
    =
    F(z) \odot_{z = 1} G(z x).
\]

\begin{lemma}
\label{th:gaussian_transform}
For any polynomial $P(x)$, we have
\[
    e^{x^2/2}
    \odot_x
    P(x)
    =
    \sum_n
    (2n-1)!!
    [z^{2n}] P(z)
    x^{2n}
    =
    \frac{1}{\sqrt{2 \pi}}
    \int_{\reals}
    P(x t)
    e^{-t^2/2}
    d t.
\]
\end{lemma}

When $\vz = (z_1, \ldots, z_k)$ is a vector of variables,
and $F_j(z)$ are series with $F_j(0) = 0$
(to ensure that $e^{F_j(z)}$ is a well defined formal power series,
see \eg \cite{kauers2011formal}),
we define the compact notation $\odot_{\vz}$ as
\[
    e^{\sum_{j=1}^k F_j(z_j)}
    \odot_{\vz}
    G(\vz)
    :=
    e^{F_1(z_1)}
    \odot_{z_1}
    \left(
    e^{F_2(z_2)}
    \odot_{z_2}
    \left(
    \cdots
    \odot_{z_{k-1}}
    \left(
    e^{F_k(z_k)}
    \odot_{z_k}
    G(\vz)
    \right)
    \cdots
    \right)
    \right).
\]
The notation $\odot_{\vz = \vx}$
is defined in the same way,
replacing each $\odot_{z_j}$
with $\odot_{z_j = x_j}$.

\begin{corollary}
\label{th:had_to_int}
For any multinomial $P(x_1, \ldots, x_k)$
and positive values $\alpha_1, \ldots, \alpha_k$,
we have
\[
    e^{\sum_{j=1}^k \alpha_j x_j^2 / 2}
    \odot_{\vx = \vone}
    P(\vx)
    =
    \frac{1}{(2 \pi)^{k/2} \sqrt{\alpha_1 \cdots \alpha_k}}
    \int_{\reals^k}
    P(\vt) e^{- \sum_{j=1}^k t_j^2 / (2 \alpha_j)}
    d \vt.
\]
\end{corollary}

\paragraph{Asymptotic expansion.}
The theory of asymptotic series was created independently
by Poincaré and Stieltjes (\cite{jahnke2003history}).
Consider functions $\varphi_j(n)$
satisfying
$\varphi_{j+1}(n) \underset{n \to +\infty}{=} \smallo(\varphi_j(n))$
for all $j$.
The notation
\[
    a_n \approx
    \sum_{j \geq 0} c_j \varphi_j(n)
\]
then means that for any $r \geq 0$, we have
\[
    a_n =
    c_0 \varphi_0(n) + \cdots + c_{r-1} \varphi_{r-1}(n)
    + \bigO(\varphi_r(n)).
\]
We then say that $\sum_{j \geq 0} c_j \varphi_j(n)$
is the \emph{asymptotic expansion} of $a_n$
with respect to the \emph{asymptotic scale}
$(\varphi_j(n))_{j \geq 0}$.
The mention of the asymptotic scale
is often clear from the context and omitted.
In particular, given a formal power series
$\tF(z) = \sum_{n \geq 0} f_n z^n$,
we write
\[
    a_n \approx b_n \tF(n^{-1})
\]
if for any $r \geq 0$, we have
\[
    a_n =
    b_n (f_0 + f_1 n^{-1} + \cdots + f_{r-1} n^{-r+1} + \bigO(n^{-r})).
\]
Note that this asymptotics holds for any fixed $r$,
as $n$ tends to infinity.
The series $\tF(z)$ might have a zero radius of convergence.
Thus, there is no guarantee that for a given $n$,
\[
    a_n = \lim_{r \to +\infty} b_n \sum_{k=0}^{r-1} f_k n^{-k}.
\]
Stirling approximation is a famous example of divergent asymptotic expansion
\[
    n! \approx
    n^n e^{-n} \sqrt{2 \pi n}
    \left(
    1
    + \frac{1}{12} n^{-1}
    + \frac{1}{288} n^{-2}
    - \frac{139}{51840} n^{-3}
    - \cdots
    \right)
\]
(see \eg many different proofs in \cite{FS09},
a proof due to Wench available in \cite[p.~267]{comtet2012advanced},
\cite{marsaglia_new_1990}
or the end of \cref{sec:laplace}).
An important property of asymptotic expansions
is that if $c_n$ converges exponentially fast to $0$,
then it can be added without impact to the asymptotic expansion
\[
    b_n (\tF(n^{-1}) + c_n) \approx
    b_n \tF(n^{-1}).
\]
Indeed, for each $r \geq 0$, we have $c_n = \bigO(n^{-r})$, so
\begin{align*}
    b_n (\tF(n^{-1}) + c_n)
    &=
    b_n (f_0 + f_1 n^{-1} + \cdots + f_{r-1} n^{-r+1} + \bigO(n^{-r}) + c_n)
    \\&=
    b_n (f_0 + f_1 n^{-1} + \cdots + f_{r-1} n^{-r+1} + \bigO(n^{-r})).
\end{align*}

        \section{Regular graphs}
        \label{sec:regular:graphs}

Since the sum of the degrees of a graph
is twice its number of edges,
there is no $k$-regular graph on $n$ vertices
when $n k$ is odd.
Thus, in the following,
we assume without loss of generality
that $n k$ is even.
The $k$-regular graphs for $k \in \{0,1\}$ being trivials,
we also assume $k \geq 2$.

We proved in a previous work (\cite{caizergues2023exact})
an exact expression enumerating regular graphs.
We will use the following equivalent formulation.

\begin{proposition}[\cite{caizergues2023exact}]
\label{th:sgkn}
The number $\sg^{(k)}_n$
of $k$-regular graphs with $n$ vertices
is equal to
\[
    \sg^{(k)}_n =
    (-1)^{n k / 2}
    e^{\sum_{j=1}^k x_j^2/(2j)}
    \exphad_{\vx = \vone}
    \left(
    [y^k]
    \frac{e^{- i \sum_{j=1}^k x_j y^j}}
        {\sqrt{1 - y^2}}
    \right)^n.
\]
\end{proposition}

The asymptotic expansion of $\sgkn$
will be extracted using the Laplace method,
presented in the next subsection.

    \subsection{Laplace method} \label{sec:laplace}

The Laplace method is a central tool of asymptotic analysis
that provides asymptotics for integrals of the form
\[
    \int_I A(t) e^{- n \phi(t)} dt,
\]
as $n$ tends to infinity.
This type of integrals appears commonly in probability theory and physics,
hence the notation $A(t)$ for the amplitude
and $\phi(t)$ for the phase.
Many sources present the main asymptotics
(\eg \cite{bruijn58}, \cite{FS09})
or a way to compute the asymptotic expansion
(\cite{erdelyi1956asymptotic}, \cite{Olver74}, \cite{wong2001asymptotic}, \cite{PW13}).
Some authors
(\cite{campbell1987explicit}, \cite{wojdylo2006coefficients}, \cite{nemes2013explicit})
focus on deriving explicit expressions for the coefficients.
In contrast, we prefer working with generating functions (series)
rather than sequences.
Recent works on asymptotic expansions
(\cite{Bo16}, \cite{dovgal2023asymptotics})
show that gathering the coefficients
of an asymptotic expansion into a generating function
is convenient for their manipulation.
We present a version of the Laplace method
following closely the proof by \cite{PW13},
adding only a more explicit expression
for the generating function of the coefficients
of the asymptotic expansion.

\begin{proposition}
\label{th:laplace}
Consider a compact interval $I$ neighborhood of $0$
and functions $A(x)$ and $\phi(x)$ analytic on $I$
such that $A(x)$ is not the zero function,
$\phi'(0) = 0$,
$\phi''(0) \neq 0$
and $\Real(\phi(x))$ reaches its minimum on $I$
only at $0$.
Then there exists a formal power series $\tF(z)$
such that the following asymptotic expansion holds
\[
    \int_I A(t) e^{- n \phi(t)} d t
    \approx
    e^{-n \phi(0)}
    \sqrt{\frac{2 \pi}{\phi''(0) n}}
    \tF(n^{-1}).
\]
Define the series $\psi(t)$ and $T(x)$ as
\[
    \psi(t) = \left( \frac{\phi(t) - \phi(0)}{\phi''(0) t^2/2} \right)^{-1/2},
    \qquad
    T(x) = x \psi(T(x)).
\]
then
\[
    \tF(z) =
    e^{z x^2/ (2 \phi''(0))}
    \odot_{x = 1}
    A(T(x)) T'(x)
\]
or, if one wants more explicit expressions for the coefficients,
\[
    [z^k] \tF(z) =
    \frac{(2k-1)!!}{\phi''(0)^k}
    [t^{2k}] A(t) \psi(t)^{2k+1}
\]
for all $k \geq 0$.
\end{proposition}

We chose the letter $T$ for our function $T(x)$
because equations of the form $T(x) = x \psi(T(x))$
are typical of rooted tree enumeration \cite[Section I.5]{FS09}.
Our choice of defining $\psi(t)$
as $\left( \frac{\phi(t) - \phi(0)}{\phi''(0) t^2/2} \right)^{-1/2}$
instead of $\left( \frac{\phi(t) - \phi(0)}{t^2/2} \right)^{-1/2}$
ensures that when $\phi(t)$ has rational coefficients
(which is often the case in combinatorics),
so does $T(x)$.
This avoids the unnecessary introduction of algebraic numbers
and facilitates the use of computer algebra systems.

A common variant of the previous result is
the case where $I$ is an unbounded interval
instead of a compact one.

\begin{proposition}
\label{th:variant_laplace}
In \cref{th:laplace},
let us remove the assumption that the interval $I$ is compact,
and assume it is instead of the form
$[- a, +\infty)$,
$(-\infty, a]$
or $(-\infty, +\infty)$
for some positive $a$.
Let us also assume
that $|A(t) e^{- K \phi(t)}|$
and $|e^{- K \phi(t)}|$
are integrable on $I$ for some $K \geq 0$.
Then the asymptotic expansion
from the conclusion of \cref{th:laplace} holds.
\end{proposition}

For example \cite[Proposition B1, p.~760]{FS09}, we recover the classic Stirling approximation
by applying this result to
\[
    n! =
    \int_0^{+\infty}
    t^n e^{-t} dt.
\]
After the change of variable $x = n t$, we obtain
\[
    n!
    =
    n^{n+1}
    \int_0^{+\infty}
    x^n e^{- n x} d x
    =
    n^{n+1}
    \int_0^{+\infty}
    e^{- n (x - \log(x))}
    d x,
\]
where the function $x - \log(x)$
reaches its unique minimum on $\reals_{\geq 0}$
at $x = 1$.
We introduce the function
$\phi(t) = 1 + t - \log(1 + t)$
and the change of variable $x = 1 + t$
\[
    n! =
    \int_{-1}^{+\infty}
    e^{- n \phi(t)} dt.
\]
Applying \cref{th:variant_laplace} yields
\[
    n!
    \approx
    n^n e^{-n} \sqrt{2 \pi n}
    (s_0 + s_1 n^{-1} + s_2 n^{-2} + \cdots)
\]
where for any $k \geq 0$,
\[
    s_k =
    (2k-1)!!
    [t^{2k}]
    \left( \frac{t - \log(1+t)}{t^2/2} \right)^{-k - 1/2}.
\]
A combinatorial interpretation for those coefficients
has been obtained by Wrench \cite[p.~267]{comtet2012advanced}.
Alternatively, defining
\[
    \psi(t) = \left( \frac{t - \log(1 + t)}{t^2/2} \right)^{-1/2}
    \quad \text{and} \quad
    T(x) = x \psi(T(x)),
\]
we have the formal power series representation
\[
    \tS(z) := \sum_{k \geq 0} s_k z^k =
    e^{z x^2/2} \odot_{x=1} T'(x).
\]

    \subsection{Asymptotic expansion of regular graphs}
    \label{sec:asympt:expansion:sgkn}

Using \cref{th:had_to_int},
we replace the exponential Hadamard products
in the expression of $\sgkn$ from \cref{th:sgkn}
with integrals,
then apply elementary changes of variables
to derive an expression amenable to the Laplace method.

\begin{lemma}
\label{th:sgkn:integral}
Assume $k \geq 2$ and $n k$ even,
and define successively the power series
with rational coefficients
\begin{align*}
    B_0(u, y, \vt)
    &=
    \sum_{\ell = 1}^k
    [z^{\ell}]
    \frac{
        \left(
            1
            + \frac{u}{1 + t_1}
            \left(
                \frac{k-1}{2} \frac{y z}{(1 + t_1)^2}
                + \sum_{j = 2}^k t_j z^{j-1}
            \right)
        \right)^{k-\ell}
    }{\sqrt{1 - z^2}}
    \frac{k!}{(k - \ell)!}
    \left( \frac{u y}{1 + t_1} \right)^{\ell},
\\
    B_1(u, y, \vt)
    &=
    \exp \left(
        - \frac{\log(1 + B_0(u, y, \vt)) - k (k-1) \frac{u^2 t_2 y}{(1 + t_1)^2}}
        {y^2}
    + \frac{(k-1)^2}{4 (1 + t_1)^4}
    + (2 k^2 u^2 - k + 1) \frac{k-1}{4}
    \right),
\\
    B_2(y, \vt)
    &=
    B_1 \left( - \frac{1}{\sqrt{k}}, y, \vt \right)
    + B_1 \left( \frac{1}{\sqrt{k}}, y, \vt \right),
\\
    \phi(t)
    &=
    \frac{t^2}{2} + t - \log(1 + t).
\end{align*}
Then the number of $k$-regular graphs on $n$ vertices
is equal to
\[
    \sgkn
    =
    \frac{(n k / e)^{n k / 2} \sqrt{k}}{k!^{n - 1/2}}
    e^{- (k^2 - 1) / 4}
    \left( \frac{n}{2 \pi} \right)^{k/2}
    \int_{\reals_{> - 1} \times \reals^{k-1}}
    B_2(i n^{-1/2}, \vt)
    e^{- n k \phi(t_1) - \sum_{j=2}^k n j t_j^2 / 2}
    d \vt.
\]
\end{lemma}

To get closer to the Laplace method,
our next lemma shows that the integral outside a vicinity
of the saddle-point is negligible.

\begin{lemma}
\label{th:bound:tails}
For any $k \geq 2$
and small enough neighbordhood $V \subset \reals^k$ of the origin,
there exists $\delta > 0$
such that
the number $\sgkn$ of $k$-regular graphs on $n$ vertices
is equal to
\[
    \sgkn
    =
    \frac{(n k / e)^{n k / 2} \sqrt{k}}{k!^{n - 1/2}}
    e^{- (k^2 - 1) / 4}
    \left(
    \frac{n}{2 \pi} \right)^{k/2}
    \int_V
    B_2(i n^{-1/2}, \vt)
    e^{- n k \phi(t_1) - \sum_{j=2}^k n j t_j^2 / 2}
    d \vt
    \left( 1 + \bigO(e^{-\delta n})\right),
\]
where $B_2(y,\vt)$ is defined in \cref{th:sgkn:integral}.
\end{lemma}

Since $e^{-\delta n}$ is exponentially small,
the asymptotic expansion of $\sgkn$
is the same as the asymptotic expansion
of the term before $(1 + \bigO(e^{-n \delta}))$.
The Laplace method is applied to extract it
and the terms are rearranged to express
the coefficients of the asymptotic expansion
as formal functions of $k$.

\begin{theorem}
\label{th:main:result}
Assume $k \geq 2$ and $n k$ is even,
then the asymptotic expansion
of the number of $k$-regular graphs
on $n$ vertices is
\[
    \sgkn
    \approx
    \frac{(n k / e)^{n k/2}}{k!^n}
    \frac{e^{- (k^2 - 1) / 4}}{\sqrt{2}}
    \tsgk(n^{-1})
\]
where for all $r$, the $r$th coefficient
of the formal power series $\tsgk(z)$
is a polynomial with rational coefficients
in $k$, divided by $k^r$,
explicitly computable using the formula
\begin{align*}
    \psi(t)
    &=
    \Bigg(
        1
        + \frac{
            \log \big( \frac{1}{1+t} \big)
            + t
            - \frac{t^2}{2}}
        {t^2}
    \Bigg)^{-1/2},
\\
    T(x)
    &=
    x \psi(T(x)),
\\
    u_{p,q}
    &=
    [s^p] \frac{1}{(1 + T(s))^q},
\\
    v_{p,q}(\vt)
    &=
    [z^p]
    \frac{\left( \sum_{j \geq 2} t_j z^{j-1} \right)^q}
    {\sqrt{1 - z^2}},
\\
    B_{0,j}(u, \vt)
    &=
    \sum_{\substack{1 \leq \ell,\ 0 \leq a \leq \ell,\ 0 \leq b\\ a + b + \ell \leq j}}
    \bigg(
        \prod_{m = 0}^{a + b + \ell - 1}
        (k - m)
    \bigg)
    \frac{u^{a + b + \ell} t_1^{j - a - b - \ell}}{a! b!}
    \left( \frac{k-1}{2} \right)^a
    u_{j - a - b - \ell, 3 a + b + \ell}
    v_{\ell - a, b}(\vt),
\\
    B_0(s, u, \vt)
    &=
    \sum_{j \geq 1}
    B_{0,j}(u, \vt) s^j
\\
    C_1(u, s, \vt)
    &=
    \exp \left(
        \frac{
            k (k-1) \frac{u^2 s^2 t_2}{(1 + T(s t_1))^2}
            - \log \left(
            1
            + B_0(s,u,\vt) \right)}
        {s^2}
    + \frac{(k-1)^2}{4 (1 + T(s t_1))^4}
    + (2 k^2 u^2 - k + 1) \frac{k-1}{4}
    \right),
\\
    C_2(s, \vt)
    &=
    \left(
    C_1 \left( - \frac{1}{\sqrt{k}}, s, \vt \right)
    + C_1 \left( \frac{1}{\sqrt{k}}, s, \vt \right)
    \right)
    T'(s t_1),
\\
    [z^r] \tsgk(z)
    &=
    (-1)^r
    e^{- t_1^2 / (4k) - \sum_{j=2}^{2r + 2} t_j^2 / (2j)}
    \odot_{\vt = \vone}
    [s^{2r}]
    C_2(s, \vt).
\end{align*}
\end{theorem}

    \subsection{Computations and numerical experiment}
    \label{sec:numerical}

We published sagemath code for computing
the coefficients of the formal power series $\tsg(z)$
at \cite{mygitlab}.
This repository contains more coefficients than provided here.
From a computational perspective,
the fact that the definition $T(x) = x \psi(T(x))$
is implicit is not an inconvenience.
Coefficients of this series
are computable by Lagrange inversion or,
more efficiently, by Newton iteration
\cite{pivoteau2012algorithms}.
For example, since
\[
    u_{p,q} =
    [s^p] H(T(s))
\]
with $H(x) = \frac{1}{(1 + x)^q}$
and $T(x) = x \psi(T(x))$,
Lagrange inversion implies
\[
    u_{p,q}
    =
    \frac{1}{p}
    [s^{p-1}]
    H'(s) \psi(s)^p
    =
    - \frac{q}{p}
    [s^{p-1}]
    \frac{\psi(s)^p}{(1 + s)^{q + 1}}.
\]

For the first few error terms
in the result of \cref{th:main:result},
we find (assuming $k \geq 3$ to shorten the expressions)
\begin{align*}
    [z^0] \tsgk(z) =\ 
&
    2,
\\
    [z^1] \tsgk(z) =\ 
&
    - \frac{1}{6} (k^4 - 2 k^2 + 3 k - 1)/k,
\\
    [z^2] \tsgk(z) =\ 
&
    \frac{1}{144} (k^8 - 16 k^6 + 6 k^5 + 50 k^4 + 36 k^3 - 239 k^2 + 234 k - 71)/k^2.
\end{align*}
For $k \in \{3, 4, 5\}$, we deduce
\begin{align*}
    \tsg^{(3)}(z)
&=
    2 - \frac{71}{18} z - \frac{143}{1296} z^2
    + \bigO(z^3),
\\
    \tsg^{(4)}(z)
&=
    2 - \frac{235}{24} z + \frac{18289}{2304} z^2
    + \bigO(z^3),
\\
    \tsg^{(5)}(z)
&=
    2 - \frac{589}{30} z + \frac{190249}{3600} z^2
    + \bigO(z^3).
\end{align*}
Our theorem states
\[
    \left(
    \sgkn
    \left(
        \frac{(nk/e)^{nk/2}}{k!^n}
        \frac{e^{-(k^2-1)/4}}{\sqrt{2}}
    \right)^{-1}
    -
    \sum_{j = 0}^{r-1}
    [z^j] \tsgk(z)
    n^{-j}
    \right)
    n^r
    =
    \bigO(1).
\]
To provide numerical credibility,
we check in the next table
that for $r=3$ and each $k \in \{2, 3, 4, 5\}$,
the left-hand side appears to be a bounded function of $n$
for $n \in [10,100]$.
We used the tables from \cite[A002829, A005815, A338978]{oeis}.
\[
    \begin{array}{c|cccccccccc}
n & 10 & 20 & 30 & 40 & 50 & 60 & 70 & 80 & 90 & 100
\\
\hline
k=2&
1.79 &
1.79 &
1.80 &
1.80 &
1.79 &
1.79 &
1.79 &
1.79 &
1.79 &
1.79
\\
k=3 &
5.04 &
4.05 &
3.79 &
3.66 &
3.60 &
3.55 &
3.52 &
3.50 &
3.48 &
3.46
\\
k=4 &
17.93 &
15.37 &
14.75 &
14.47 &
14.31 &
14.21 &
14.14 &
14.08 &
14.04 &
14.01
\\
k=5 &
2.16 &
3.59 &
4.36 &
4.75 &
4.98 &
5.13 &
5.24 &
5.32 &
5.38 &
5.43
\end{array}
\]

        \section{Connected regular graphs}
        \label{sec:connected:regular:graphs}

It is known since \cite{Wright70}
that the asymptotic expansion of regular graphs
gives access to the asymptotic expansion of connected regular graphs.
In this section, we detail this process.
Since a connected $2$-regular graph
is just a cycle, we will assume in this section $k \geq 3$.
Let $\csg^{(k)}_n$ denote the number
of connected $k$-regular graphs on $n$ vertices.
The associated generating function is
\[
    \csg^{(k)}(z) =
    \sum_{n \geq 0}
    \csg^{(k)}_n
    \frac{z^n}{n!}.
\]
A $k$-regular graphs is a set of connected $k$-regular graphs,
so the generating functions of $k$-regular graphs
and of connected $k$-regular graphs
are linked by the classic relation
\[
    \sg^{(k)}(z) =
    e^{\csg^{(k)}(z)},
\]
and
\begin{equation}
\label{eq:csg:log}
    \csg^{(k)}(z) =
    \log(\sg^{(k)}(z)).
\end{equation}
Notice that in the definition of
\[
    \sg^{(k)}(z) =
    \sum_{n \geq 0}
    \sg^{(k)}_n
    \frac{z^n}{n!},
\]
we admit the empty graph with $n = 0$ vertices.
This ensures that $\sg^{(k)}(0) = 1$,
so \cref{eq:csg:log} properly characterizes
the formal power series $\csg^{(k)}(z)$.

Given the asymptotics of the coefficients of $\sgkn$,
the associated generating function $\sg^{(k)}(z)$
has a zero radius of convergence.
Our next section presents tools for extracting
asymptotic expansion of the coefficients
of such divergent series,
based on the work of \cite{Wright70} and \cite{Be75}.

    \subsection{Divergent series}
    \label{sec:divergent:series}

Let us first recall a key tool for the computation
of the asymptotic expansion of fast growing coefficients
of formal power series, due to \cite{Wright70} and \cite{Be75}.

\begin{theorem}[\cite{Wright70}, \cite{Be75}]
\label{th:divergent:series}
Consider a function $H(z)$ analytic at $z=0$,
a positive integer $R$
and a formal power series
\[
    A(z) =
    \sum_{n > 0}
    a_n z^n
\]
whose coefficients satisfy
$a_n \neq 0$ for all sufficiently large $n$,
$a_{n-1} = \smallo(a_n)$,
and
\[
    \sum_{j = R}^{n - R} a_j a_{n-j} = \bigO(a_{n-R}),
\]
then
\[
    [z^n] H(A(z)) =
    \sum_{j = 0}^{R - 1}
    c_j a_{n-j}
    + \bigO(a_{n-R}),
\]
where
\[
    c_j =
    [z^j]
    H'(A(z)).
\]
\end{theorem}

In order to use the previous theorem,
we have to check that $\sgkn$
satisfies the hypothesis
\[
    \sum_{j = R}^{n - R} \sgk_j \sgk_{n-j}
    = \bigO(\sgk_{n-R}).
\]
The next lemma simplifies this task,
requiring only information
on the main asymptotics of $\sgkn$.

\begin{lemma}
\label{th:alpha:beta:gamma}
Consider positive $\alpha$, $\beta$, a real value $\gamma$
and a positive sequence $(a_n)_{n > 0}$ satisfying
$a_n = \exactbigO(n^{\alpha} \beta^n n^{\gamma})$,
then for any fixed $R \in \integers_{> 0}$,
as $n$ tends to infinity, we have
\[
    \sum_{j=R}^{n-R} a_j a_{n-j} =
    \bigO(a_{n-R}).
\]
\end{lemma}

The following lemma provides the asymptotic expansion
of $a_{n-j}$ given the asymptotic expansion of $a_n$,
provided it is of a special form
(satisfied by $\sgkn$).

\begin{lemma}
\label{th:an_j}
Consider $\alpha \in \integers_{> 0}$,
$\beta \in \reals_{> 0}$
and $\gamma \in \reals$,
and a sequence $(a_n)_n$
with asymptotic expansion
\[
    a_n \approx
    n^{\alpha n}
    \beta^n
    n^{\gamma}
    \tA(n^{-1})
\]
for some nonzero formal power series $\tA(z)$.
Define the formal power series
\[
    \tA_j(z) =
    e^{- \alpha j}
    \beta^{-j}
    z^{\alpha j}
    (1 - j z)^{\gamma - \alpha j}
    e^{\alpha z^{-1} (\log(1 - j z) + j z)}
    \tA \left( \frac{z}{1 - j z} \right).
\]
Then for any fixed $j \in \integers_{\geq 0}$,
we have as $n$ tends to infinity
\[
    a_{n-j} \approx
    n^{\alpha n}
    \beta^n
    n^{\gamma}
    \tA_j(n^{-1}).
\]
\end{lemma}

Applying \cref{th:alpha:beta:gamma}
and \cref{th:an_j}
to \cref{th:divergent:series},
we obtain the tool we will apply
to extract the asymptotic expansion of $\csgk_n$
for $k$ even.

\begin{proposition}
\label{th:HA}
Consider a function $H(z)$ analytic at $0$
and a formal power series
\[
    A(z) = \sum_{n > 0} a_n z^n
\]
whose coefficients satisfy
\[
    a_n \approx
    n^{\alpha n}
    \beta^n
    n^{\gamma}
    \tA(n^{-1})
\]
for some $\alpha \in \integers_{> 0}$,
$\beta \in \reals_{> 0}$,
$\gamma \in \reals$
and nonzero formal power series $\tA(z)$,
then
\[
    [z^n] H(A(z)) \approx
    n^{\alpha n}
    \beta^n
    n^{\gamma}
    \tA_H(n^{-1})
\]
where the formal power series $\tA_H(z)$
is defined as
\begin{align*}
    \tA_H(z) &=
    \sum_{j \geq 0}
    \tA_j(z)
    [x^j] H'(A(x)),
\\
    \tA_j(z) &=
    e^{- \alpha j}
    \beta^{-j}
    z^{\alpha j}
    (1 - j z)^{\gamma - \alpha j}
    e^{\alpha z^{-1} (\log(1 - j z) + j z)}
    \tA \left( \frac{z}{1 - j z} \right).
\end{align*}
\end{proposition}

Observe that $\tA_j(z)$ has valuation
at least $z^{\alpha j}$,
so each coefficient of $\tA_H(z)$
is computable with a finite sum
\[
    [z^{\ell}] \tA_H(z) =
    \sum_{j=0}^{\lfloor \ell / \alpha \rfloor}
    [z^{\ell}]
    \tA_j(z)
    [x^j]
    H'(A(x)).
\]

When $k$ is odd, $\sgkn = 0$ for all odd $n$.
Thus, the assumption from \cref{th:divergent:series}
that the coefficients are eventually positive
is not satisfied in that case.
However, as shown in the next corollary,
we can work around this difficulty.
Indeed, for $k$ odd, we can work with $\sgk(\sqrt{z})$
instead of $\sgk(z)$.
It is a formal power in $z$
with coefficients eventually positive,
so our last proposition is applicable.

\begin{corollary}
\label{th:divergent:even}
Consider a function $H(z)$ analytic at $0$
and a formal power series
\[
    B(z) = \sum_{n > 0} b_n z^n
\]
whose coefficients satisfy $b_n = 0$
for all odd $n$, and for $n$ even
\[
    b_n \approx
    n^{\alpha n}
    \beta^n
    n^{\gamma}
    \tB(n^{-1})
\]
for some $\alpha \in \frac{1}{2} \integers_{> 0}$,
$\beta \in \reals_{> 0}$,
$\gamma \in \reals$
and nonzero formal power series $\tB(z)$.
Then for $n$ even
\[
    [z^n] H(B(z)) \approx
    n^{\alpha n}
    \beta^n
    n^{\gamma}
    \tB_H(n^{-1})
\]
where the formal power series $\tB_H(z)$
is defined as
\begin{align*}
    \tB_H(z) &=
    \sum_{j \geq 0}
    \tB_{2j}(z)
    [x^{2j}] H'(B(x)),
\\
    \tB_j(z) &=
    e^{- \alpha j}
    \beta^{-j}
    z^{\alpha j}
    (1 - j z)^{\gamma - \alpha j}
    e^{\alpha z^{-1} (\log(1 - j z) + j z)}
    \tB \left( \frac{z}{1 - j z} \right).
\end{align*}
\end{corollary}

    \subsection{Asymptotic expansion of connected regular graphs}
    \label{sec:asympt:csgkn}

Applying \cref{th:HA} (if $k$ is even)
and \cref{th:divergent:even} (if $k$ is odd)
to the expression of $\csgkn$
from \cref{eq:csg:log},
we finally express the asymptotic expansion
of connected regular graphs.

\begin{theorem}
\label{th:csgkn}
For any $k \geq $,
the number $\csgkn$ of connected $k$-regular graphs
on $n$ vertices has asymptotic expansion
\[
    \csgkn \approx
    \frac{(nk/e)^{nk/2}}{k!^n}
    \frac{e^{-(k^2-1)/4}}{\sqrt{2}}
    \tcsgk(n^{-1})
\]
where the formal power series $\tcsgk(z)$
is computed using the following equations
\begin{align*}
    \psi(t) &=
    \left( \frac{t - \log(1+t)}{t^2/2} \right)^{-1/2},
\\
    T(x) &= x \psi(T(x)),
\\
    \tS(z) &= e^{z x^2/2} \odot_{x=1} T'(x),
\\
    f_{k,j}(z) &=
    \begin{cases}
        1 & \text{if } j = 0,
    \\
        \sum_{\ell \geq 3}
        \indic_{k = \ell}
        \indic_{j\, \ell \text{ even}}
        z^{(\ell/2-1) j}
    & \text{if } j > 0,
    \end{cases}
\\
    \tA^{(k)}(z) &= \frac{\tsgk(z)}{\tS(z)},
\\
    \tA^{(k)}_j(z) &=
    \left( \frac{k!}{k^{k/2}} \right)^j
    f_{k,j}(z)
    (1 - j z)^{-1/2 - (k/2 - 1) j}
    e^{(k/2 - 1) z^{-1} (\log(1 - j z) + j z)}
    \tA^{(k)} \left( \frac{z}{1 - j z} \right),
\\
    \tcsgk(z) &=
    \tS(z)
    \sum_{j \geq 0}
    \tA^{(k)}_j(z)
    [x^j]
    \frac{1}{\sgk(x)}.
\end{align*}
\end{theorem}

Since $\sgkn = 0$ for all $n \in [2, k]$, we have
\[
    \tcsgk(z) =
    \tS(z) \tA^{(k)}_0(z) +
    \sum_{j \geq k+1}
    \tA^{(k)}_j(z)
    [x^j] \frac{1}{\sgk(x)}.
\]
We replace $\tA^{(k)}_0(z)$ with its expression
$\frac{\tsg^{(k)}(z)}{\tS(z)}$.
Observe that for any $j \geq k+1$,
the valuation of $\tA^{(k)}_j(z)$
is equal to the valuation of $f_{k,j}(z)$,
which is $(k+1)(k-2)/2$
for $j = k+1$, and greater for $j > k+1$.
Thus,
\[
    \tcsgk(z) =
    \tsg^{(k)}(z) +
    \exactbigO \left(z^{(k+1)(k-2)/2} \right).
\]
This implies the following theorem,
a more precise version of the result
of \cite{Bollobas} and \cite{wormald1981asymptotic},
stating that $\sgkn$ and $\csgkn$
have the same asymptotics.

\begin{theorem}
\label{th:link:asympt}
For any fixed $k \geq 3$,
the numbers $\sgkn$ and $\csgkn$
of $k$-regular graphs and connected $k$-regular graphs
on $n$ vertices are linked by the relation
\[
    \csgkn =
    \sgkn
    \left(1 + \exactbigO(n^{- (k+1)(k-2)/2}) \right).
\]
\end{theorem}

    \subsection{Computations and numerical experiment}
    \label{sec:connected:numerical}

We published sagemath code for computing
the coefficients of the formal power series $\tcsg(z)$
at \cite{mygitlab}.
This repository contains more coefficients than provided here.
Let us call \emph{truncation of order $r$}
of the formal power series $F(z)$
the series
\[
    \sum_{j = 0}^r [x^j] F(x) z^j.
\]
With the notations of \cref{th:csgkn},
to express the error term $[z^r] \tcsgk(z)$,
we successively compute
\begin{enumerate}
\item
the truncation of order $2r$ of $\psi(t)$,
\item
the truncation of order $2r+1$ of $T(x)$,
\item
the truncation of order $r$ of $\tS(z)$,
\item
the truncation of order $r$ of $f_{k,j}(z)$,
where the sum can be stopped at $\ell = 2 r + 2$,
\item
the truncation of order $r$ of $\tA^{(k)}(z)$,
\item
the truncation of order $r$ of $\tA^{(k)}_j(z)$
for $j$ in $[0, 2r]$,
\item
the truncation of order $r$ of $\tcsgk(z)$,
where the sum can be stopped at $j = 2r$.
\end{enumerate}
This last point is justified by the observation
that $f_{k,j}(z)$ has valuation at least $j/2$
for any $k \geq 3$,
so $\tA^{(k)}_j(z)$ has valuation at least $j/2$.
The result is a function of $k$ and
$\sgk_n$ for $n \in [4, 2r]$.
We find for the first few error terms from \cref{th:csgkn}
\begin{align*}
    [z^0] \tcsgk(z) =\
& 
    2,
\\
    [z^1] \tcsgk(z) =\ 
&
-\frac{1}{6} (k^4 - 2 k^2 + 3 k - 1) k^{-1},
\\
    [z^2] \tcsgk(z) =\ 
&
\frac{1}{144} (k^8 - \indic_{k=3} 12 (k!)^4 k^{2-2k} \sgk_4 - 16 k^6 + 6 k^5 + 50 k^4 + 36 k^3 - 239 k^2 + 234 k - 71) k^{-2}.
\end{align*}
We used the assumption $k \geq 3$
to deduce $\sgkn = 0$ for $n \in \{1, 2, 3\}$
and reduce the expressions of the coefficients.
In particular, for $k \in \{3, 4, 5\}$, we have
\begin{align*}
    \tcsg^{(3)}(z)
&=
    2 - \frac{71}{18} z - \frac{335}{1296} z^2 
    + \bigO(z^3),
\\
    \tcsg^{(4)}(z)
&=
    2 - \frac{235}{24} z + \frac{18289}{2304} z^2
    + \bigO(z^3),
\\
    \tcsg^{(5)}(z)
&=
    2 - \frac{589}{30} z + \frac{190249}{3600} z^2
    + \bigO(z^3).
\end{align*}
Observe that, as predicted by \cref{th:link:asympt},
the only difference between
$\tcsg^{(k)}(z)$ and $\tsg^{(k)}(z)$ for $k \in \{3,4,5\}$
is the term $[z^2] \tcsg^{(3)}(z)$.
\cref{th:csgkn} states
\[
    \left(
        \csgkn
        \left(
            \frac{(n k / e)^{n k / 2}}{k!^n}
            \frac{e^{-(k^2-1)/4}}{\sqrt{2}}
        \right)^{-1}
        - \sum_{j=0}^{r-1}
        [z^j] \tcsgk(n^{-1})
    \right)
    n^r
    =
    \bigO(1).
\]
As in \cref{sec:numerical},
we check that for $r=3$ and each $k \in \{3, 4\}$,
the left-hand side appears to be a bounded function of $n$
for $n \in [10,100]$.
We used the tables from \cite[A004109, A272905]{oeis}.
\[
    \begin{array}{c|cccccccccc}
n & 10 & 20 & 30 & 40 & 50 & 60 & 70 & 80 & 90 & 100
\\
\hline
k=3 &
4.40 &
2.05 &
2.15 &
2.26 &
2.30 &
2.31 &
2.31 &
2.31 &
2.31 &
2.31
\\
k=4 &
17.93 &
15.37 &
14.75 &
14.47 &
14.31 &
14.20 &
14.14 &
14.08 &
14.04 &
14.01
\end{array}
\]

        \section{Proofs} \label{sec:proofs}

This section gathers the proofs omitted in the main document,
and a few technical additional results.
The numbers in parenthesis point to the references
in the main document.

    \subsection{Proofs of \cref{sec:regular:graphs}}

\begin{proposition*}[\ref{th:sgkn}]
The number $\sg^{(k)}_n$
of $k$-regular graphs with $n$ vertices
is equal to
\[
    \sg^{(k)}_n =
    (-1)^{n k / 2}
    e^{\sum_{j=1}^k x_j^2/(2j)}
    \exphad_{\vx = \vone}
    \left(
    [y^k]
    \frac{e^{- i \sum_{j=1}^k x_j y^j}}
        {\sqrt{1 - y^2}}
    \right)^n.
\]
\end{proposition*}

\begin{proof}
We start with the formula from \cite{caizergues2023exact}
\[
    \sg^{(k)}_n =
    e^{\sum_{j=1}^k (-1)^{j+1} x_j^2 / (2 j)}
    \exphad_{\vx = \vone}
    \left(
    [y^k]
    \frac{e^{\sum_{j=1}^k x_j y^j}}
        {\sqrt{1 + y^2}}
    \right)^n.
\]
and apply first the identity
$A(\alpha x) \odot_x B(x) = A(x) \odot_x B(\alpha x)$
for each $x_j$ with $\alpha = i^{(j-1)/2}$
\[
    \sg^{(k)}_n =
    e^{\sum_{j=1}^k x_j^2/(2j)}
    \exphad_{\vx = \vone}
    \left(
    [y^k]
    \frac{e^{\sum_{j=1}^k i^{j-1} x_j y^j}}
        {\sqrt{1 + y^2}}
    \right)^n,
\]
then the identity $[y^k] A(\alpha y) = \alpha^k [y^k] A(y)$
with $\alpha = i$
\[
    \sg^{(k)}_n =
    e^{\sum_{j=1}^k x_j^2/(2j)}
    \exphad_{\vx = \vone}
    \left(
    i^k
    [y^k]
    \frac{e^{- i \sum_{j=1}^k x_j y^j}}
        {\sqrt{1 - y^2}}
    \right)^n.
\]
\end{proof}

    \subsubsection{Proofs of \cref{sec:laplace}}

We start by recalling a classical lemma to estimate Gaussian-like integrals.
It will be useful for our proof of the Laplace method.

\begin{lemma}
\label{th:double_factorial_integral}
Consider a nonnegative integer $j$,
nonzero complex numbers $\alpha$ and $\beta$
whose arguments are in $(-\pi/4, \pi/4)$, then
\[
    \lim_{n \to +\infty}
    \int_{- \alpha n}^{\beta n}
    z^j e^{- z^2/2} d z
    =
    \begin{cases}
        0 & \text{if $j$ is odd},\\
        \sqrt{2 \pi} (2 \ell - 1)!! & \text{if $j = 2 \ell$.}
    \end{cases}
\]
and the convergence is exponentially fast in $n$.
\end{lemma}

\begin{proof}
Define the positive real numbers
$a = \Real(\alpha)$ and $b = \Real(\beta)$.
We first prove that the integral from $b n$ to $\beta n$ is negligible.
Indeed, bounding the absolute value of the integral
by the length of the integration path
multiplied by the maximal value of the absolute of the integrand,
we obtain
\[
    \left|
    \int_{b n}^{\beta n}
    z^j e^{- z^2 / 2} d z
    \right|
    \leq
    \Im(\beta n)
    |n \beta|^j
    e^{- \Real((\beta n)^2 / 2)}
    \leq
    n^{j +  1}
    |\beta|^{j+1}
    e^{- n^2 \cos(2 \arg(\beta)) |\beta|^2 / 2}.
\]
Since $2 \arg(\beta)$ is in $(-\pi/2, \pi/2)$,
the cosine is positive
so this upper bound converges exponentially fast to $0$.
By symmetry, we also have
\[
    \left|
    \int_{- \alpha n}^{- a n}
    z^j e^{- z^2 / 2} d z
    \right|
\]
converging exponentially fast to $0$.

Now consider the integral on $[b n, +\infty)$.
For any $x \geq 1$, we have
\[
    x^j
    \leq
    j! 4^j e^{x/4}
    \quad \text{and} \quad
    e^{- x^2 / 2} \leq e^{- x / 2}
\]
so for any $n \geq 1/b$
\[
    \int_{b n}^{+\infty}
    x^j e^{- x^2/2} d x
    \leq
    j! 4^j
    \int_{b n}^{+\infty}
    e^{- x/4} d x
    \leq
    j! 4^{j+1}
    e^{- b n / 4}
\]
which converges exponentially fast to $0$.
Similarly, so does the integral
\[
    \int_{-\infty}^{- a n}
    x^j e^{- x^2/2} d x.
\]
By analyticity, the integration path from $- \alpha n$ to $\beta n$
can be any continuous piece-wise derivable path
linking $- \alpha n$ to $\beta n$.
Adding and subtracting parts of the integration paths,
we deduce that
\[
    \left|
    \int_{- \alpha n}^{\beta n}
    z^j e^{- z^2 / 2} d z
    - \int_{-\infty}^{+\infty}
    x^j e^{- x ^2 / 2} d x
    \right|
\]
is exponentially small.

By symmetry, for any odd $j$, we have
\[
    \int_{-\infty}^{+\infty}
    x^j e^{-x^2/2} d x
    = 0.
\]
For any even positive $j$,
integration by parts provides
\[
    \int_{-\infty}^{+\infty}
    x^j e^{-x^2/2} d x
    =
    (j-1)
    \int_{-\infty}^{+\infty}
    x^{j-2} e^{-x^2/2} d x.
\]
Following an induction
initialized with the Gaussian integral,
we deduce
\[
    \int_{-\infty}^{+\infty}
    x^j e^{-x^2/2} d x
    =
    \sqrt{2 \pi} (j - 1)!!,
\]
which concludes the proof.
\end{proof}

We recall the proof of the Laplace method for completeness
and to have the exact expression we will need,
as many variants exist in the literature.

\begin{proposition*}[\ref{th:laplace}]
Consider a compact interval $I$ neighborhood of $0$
and functions $A(x)$ and $\phi(x)$ analytic on $I$
such that $A(x)$ is not the zero function,
$\phi'(0) = 0$,
$\phi''(0) \neq 0$
and $\Real(\phi(x))$ reaches its minimum on $I$
only at $0$.
Then there exists a formal power series $F(z)$
such that the following asymptotic expansion holds
\[
    \int_I A(t) e^{- n \phi(t)} d t
    \approx
    e^{-n \phi(0)}
    \sqrt{\frac{2 \pi}{\phi''(0) n}}
    F(n^{-1}).
\]
Define the series $\psi(t)$ and $T(x)$ as
\[
    \psi(t) = \left( \frac{\phi(t) - \phi(0)}{\phi''(0) t^2/2} \right)^{-1/2},
    \qquad
    T(x) = x \psi(T(x)).
\]
then
\[
    F(z) =
    e^{z x^2/ (2 \phi''(0))}
    \odot_{x = 1}
    A(T(x)) T'(x)
\]
or, if one wants more explicit expressions for the coefficients,
\[
    [z^k] F(z) =
    \frac{(2k-1)!!}{\phi''(0)^k}
    [t^{2k}] A(t) \psi(t)^{2k+1}
\]
for all $k \geq 0$.
\end{proposition*}

\begin{proof}
We closely the proof by \cite{PW13},
adding only a more explicit expression
for the generating function of the coefficients
of the asymptotic expansion.
We start by putting the constant term of $\phi(t)$
out of the integral
\[
    \int_I A(t) e^{- n \phi(t)} d t
    =
    e^{-n \phi(0)}
    \int_I A(t) e^{- n (\phi(t) - \phi(0))} d t.
\]
For any small enough $\epsilon > 0$, we have
\[
    \int_{I \setminus [-\epsilon, \epsilon]}
    A(t) e^{- n (\phi(t) - \phi(0))} d t
    \leq
    e^{- n (\min(\phi(-\epsilon), \phi(\epsilon)) - \phi(0))}|I| \sup_{t \in I} |A(t)|.
\]
Since $\phi(t) - \phi(0)$ reaches its minimum
on the compact set $I$ only at $0$,
the value $\min(\phi(-\epsilon), \phi(\epsilon)) - \phi(0)$
is strictly positive, so the left hand-side
converges to $0$ exponentially fast.
Thus, for any small enough $\epsilon > 0$,
there exists a positive $\delta$ such that
\[
    \int_I
    A(t) e^{- n \phi(t)} d t
    =
    e^{- n \phi(0)}
    \left(
    \int_{-\epsilon}^{\epsilon}
    A(t) e^{- n (\phi(t) - \phi(0))} d t
    + \bigO(e^{- \delta n})
    \right).
\]
Our strategy to estimate the integral
is to apply a change of variable to turn it into a Gaussian integral.
Thus, we are seeking a change of variable satisfying
\[
    \phi(t) - \phi(0) = \phi''(0) \frac{x^2}{2}.
\]
It is sufficient to have
\[
    x = t \sqrt{ \frac{\phi(t) - \phi(0)}{\phi''(0) t^2 / 2} }.
\]
Since by Taylor's Theorem
\[
    \phi(t) = \phi(0) + \phi''(0) \frac{t^2}{2} + \bigO(t^3),
\]
the function $\frac{\phi(t) - \phi(0)}{\phi''(0) t^2 / 2}$
is a power series analytic at $0$ with constant term $1$.
Taking the principal value of the square root
and defining
\[
    \psi(t) =
    \left( \frac{\phi(t) - \phi(0)}{\phi''(0) t^2 / 2} \right)^{-1/2},
\]
which is analytic at $0$,
it is sufficient for our change of variable to find $T(x)$
invertible in a neighborhood of $0$ such that
\[
    T(x) = x \psi(T(x)).
\]
By the analytic inversion theorem,
this equation indeed characterizes the function $T(x)$
which is analytic and invertible in a neighborhood of $0$,
with $T(0) = 0$.
We choose $\epsilon$ small enough
to ensure that the compositional inverse of $T(x)$
is analytic and invertible on $[-\epsilon, \epsilon]$
and obtain
\[
    \int_{-\epsilon}^{\epsilon}
    A(t) e^{- n (\phi(t) - \phi(0))} d t
    =
    \int_{T^{-1}(-\epsilon)}^{T^{-1}(\epsilon)}
    A(T(x)) T'(x) e^{- n \phi''(0) x^2/2} d x
\]
where the integration path links the complex points
$T^{-1}(-\epsilon)$ and $T^{-1}(\epsilon)$
(analycity ensures that all continuous piece-wise derivable paths
staying in the domain of analycity
give the same value for the integral).

Observe that if the integration path linked $- \infty$ to $+ \infty)$
instead of $T^{-1}(-\epsilon)$ to $T^{-1}(\epsilon)$,
and $A(T(x)) T'(x)$ was a polynomial,
applying \cref{th:had_to_int} would yield
\[
    \int_{T^{-1}(-\epsilon)}^{T^{-1}(\epsilon)}
    A(T(x)) T'(x) e^{- n \phi''(0) x^2/2} d x
    =
    \sqrt{\frac{2 \pi}{\phi''(0) n}}
    e^{x^2 / (2 \phi''(0) n)}
    \odot_{x=1}
    A(T(x)) T'(x)
\]
and we would conclude
\[
    \int_I A(t) e^{- n \phi(t)} dt
    \overset{?}{=}
    e^{- n \phi(0)}
    \sqrt{\frac{2 \pi}{\phi''(0) n}}
    e^{x^2 / (2 \phi''(0) n)}
    \odot_{x=1}
    A(T(x)) T'(x).
\]
This equality is formally wrong.
However, it is the same result as the current proposition
when the equality is replaced by an asymptotic expansion.
The rest of the proof provides rigor
to this algebraic intuition.

The variable change $n \phi''(0) x \mapsto x$ is applied.
\[
    \int_{\sqrt{\phi''(0) n} T^{-1}(-\epsilon)}^{\sqrt{\phi''(0) n} T^{-1}(\epsilon)}
    A \left( T \left( \frac{x}{\sqrt{\phi''(0) n}} \right) \right)
    T'\left( \frac{x}{\sqrt{\phi''(0) n}} \right)
    e^{- x^2/2}
    \frac{d x}{\sqrt{\phi''(0) n}}.
\]
Fix a positive integer $r$.
The series $A(T(x)) T'(x)$ is replaced
by its Taylor expansion of order $2 r$
\[
    \int_{\sqrt{\phi''(0) n} T^{-1}(-\epsilon)}^{\sqrt{\phi''(0) n} T^{-1}(\epsilon)}
    \left(
        \sum_{j = 0}^{2 r - 1}
        [z^j] A(T(z)) T'(z)
        \left( \frac{x}{\sqrt{\phi''(0) n}} \right)^j
        + \bigO(x^{2r} n^{-r})
    \right)
    e^{- x^2/2}
    \frac{d x}{\sqrt{\phi''(0) n}}
\]
which is rewritten
\begin{align*}
    \frac{1}{\sqrt{\phi''(0) n}}
    \Bigg(
    &
    \sum_{j = 0}^{2 r - 1}
    [z^j] A(T(z)) T'(z)
    (\phi''(0) n)^{-j/2}
    \int_{\sqrt{\phi''(0) n} T^{-1}(-\epsilon)}^{\sqrt{\phi''(0) n} T^{-1}(\epsilon)}
    x^{2j} e^{-x^2/2} dx
    \\ & +
    n^{-r}
    \int_{\sqrt{\phi''(0) n} T^{-1}(-\epsilon)}^{\sqrt{\phi''(0) n} T^{-1}(\epsilon)}
    \bigO(x^{2r})
    e^{-x^2/2} dx
    \Bigg)
\end{align*}
Since $\Real(\phi(t))$ is minimal at $0$,
we have $\Real(\phi''(0)) > 0$
so $\arg(\sqrt{\phi''(0)}) \in (- \pi/4, \pi/4)$.
From $T(x) = x \psi(T(x))$, we deduce $T(x) = x + \bigO(x^2)$
so $T^{-1}(x) = x + \bigO(x^2)$
(this also follows from $T^{-1}(x) = \frac{x}{\psi(x)}$).
We can choose $\epsilon$ small enough
to ensure that both $T^{-1}(\epsilon)$ and $-T^{-1}(-\epsilon)$
have arguments close to $0$,
so both
$\arg(\sqrt{\phi''(0)} T^{-1}(\epsilon))$
and $\arg(- \sqrt{\phi''(0)} T^{-1}(- \epsilon))$
are in $(- \pi / 4, \pi / 4)$.
According to \cref{th:double_factorial_integral},
this implies for the second term the bound
\[
    \left|
    \int_{\sqrt{\phi''(0) n} T^{-1}(-\epsilon)}^{\sqrt{\phi''(0) n} T^{-1}(\epsilon)}
    n^{-r}
    \bigO(x^{2r}) e^{- x^2/2} \frac{d x}{\sqrt{n}}
    \right|
    =
    \bigO(n^{-r-1/2})
    \int_{-\infty}^{+\infty}
    x^{2r} e^{-x^2/2} d x
    =
    \bigO(n^{-r-1/2}).
\]
For each $j \in [1, 2 r - 1]$, according to the same lemma,
the integral
\[
    \int_{\sqrt{\phi''(0) n} T^{-1}(-\epsilon)}^{\sqrt{\phi''(0) n} T^{-1}(\epsilon)}
    x^{2j} e^{-x^2/2} dx
\]
converges super-polynomially fast
to $0$ if $j$ is odd,
and to $\sqrt{2 \pi} (2\ell - 1)!!$ if $j = 2\ell$.
Thus,
\[
    \sum_{j=0}^{2r-1}
    \int_{\sqrt{n} T^{-1}(-\epsilon)}^{\sqrt{n} T^{-1}(\epsilon)}
    [z^j] A(T(z)) T'(z)
    (\phi''(0) n)^{-j/2}
    x^j e^{- x^2/2} d x
\]
converges super-polynomially fast to
\[
    \sqrt{2 \pi}
    \sum_{\ell=0}^{r-1}
    (2\ell-1)!!
    [z^{2 \ell}] A(T(z)) T'(z)
    (\phi''(0) n)^{-\ell}.
\]
Gathering the last results, for any positive $r$, we have
\[
    \int_I A(t) e^{- n \phi(t)} d t
    \approx
    e^{-n \phi(0)}
    \sqrt{\frac{2 \pi}{\phi''(0) n}}
    \bigg(
    \sum_{\ell=0}^{r-1}
    (2\ell-1)!!
    [z^{2 \ell}] A(T(z)) T'(z)
    (\phi''(0) n)^{-\ell}
    + \bigO(n^{-r})
    \bigg).
\]
Defining
\[
    F(z)
    =
    e^{z x^2 / (2 \phi''(0))}
    \odot_{x = 1}
    A(T(x)) T'(x)
    =
    \sum_{\ell \geq 0}
    (2\ell-1)!!
    [z^{2 \ell}] A(T(z)) T'(z)
    \left( \frac{z}{\phi''(0)} \right)^\ell
\]
we deduce the asymptotic expansion
\[
    \int_I A(t) e^{- n \phi(t)} d t
    \approx
    e^{-n \phi(0)}
    \sqrt{\frac{2 \pi}{\phi''(0) n}}
    F(n^{-1}).
\]

To conclude the proof, we rewrite
$[z^{2 \ell}] A(T(z)) T'(z)$
using Lagrange inversion.
In the Cauchy integral representation
of the coefficient extraction
\[
    [z^{2 \ell}] A(T(z)) T'(z)
    =
    \frac{1}{2 i \pi} \oint
    A(T(z)) T'(z)
    \frac{d z}{z^{2\ell+1}}
\]
the variable change $T(z) = t$ is applied
\[
    [z^{2 \ell}] A(T(z)) T'(z)
    =
    \frac{1}{2 i \pi} \oint
    A(t)
    \left( \frac{\psi(t)}{t} \right)^{2\ell+1}
    d t
\]
and we recognize the Cauchy integral representation
of the coefficient extraction
\[
    [t^{2 \ell}]
    A(t)
    \psi(t)^{2\ell+1}.
\]
\end{proof}

\begin{proposition*}[\ref{th:variant_laplace}]
In \cref{th:laplace},
let us remove the assumption that the interval $I$ is compact,
so it is either of the form
$[- a, +\infty)$,
$(-\infty, a)$
or $(-\infty, +\infty)$.
Let us also assume
that $|A(t) e^{- K \phi(t)}|$
and $|e^{- K \phi(t)}|$
are integrable on $I$ for some positive $K$.
Then the asymptotic expansion
from the conclusion of \cref{th:laplace} holds.
\end{proposition*}

\begin{proof}
We consider the case $I = [-a, +\infty)$,
the other cases having similar proofs.
Since $|A(t) e^{- K \phi(t)}|$ is integrable
on $I$, it is bounded.
Since $|e^{- K \phi(t)}|$ is integrable,
its limit as $t$ tends to infinity must be $0$,
so the limit of $\phi(t)$ is $+\infty$.
By assumption, $\phi(t)$ reaches its unique minimum on $I$ at $t = 0$,
so for any $b > 0$
\[
    \inf_{t \geq b} \Real(\phi(t) - \phi(0)) > 0.
\]
Let us denote this value by $\delta$,
then
\[
    \left|
    \int_b^{+\infty}
    A(t) e^{- n (\phi(t) - \phi(0))}
    dt
    \right|
    \leq
    \int_b^{+\infty}
    |A(t) e^{- K (\phi(t) - \phi(0))}|
    e^{- (n - K) \delta}
    dt
    =
    \bigO(e^{- n \delta}).
\]
Thus, the integration on $[b,+\infty)$
is exponentially small
compared to the integration on $[-a, b]$.
This exponential difference is invisible
in the asymptotic expansion.
Applying \cref{th:laplace}, we obtain
\[
    \int_{-a}^{+\infty}
    A(t) e^{- n \phi(t)} dt
    =
    \int_{-a}^b A(t) e^{- n \phi(t)} dt
    +
    e^{- n \phi(0)}
    \bigO(e^{- n \delta})
    \approx
    e^{- n \phi(0)}
    \sqrt{\frac{2 \pi}{\phi''(0) n}}
    F(n^{-1}).
\]
\end{proof}

    \subsubsection{Proofs of \cref{sec:asympt:expansion:sgkn}}

\begin{lemma*}[\ref{th:sgkn:integral}]
Assume $k \geq 2$ and $n k$ even,
and define successively the power series
with rational coefficients
\begin{align*}
    B_0(u, y, \vt)
    &=
    \sum_{\ell = 1}^k
    [z^{\ell}]
    \frac{
        \left(
            1
            + \frac{u}{1 + t_1}
            \left(
                \frac{k-1}{2} \frac{y z}{(1 + t_1)^2}
                + \sum_{j = 2}^k t_j z^{j-1}
            \right)
        \right)^{k-\ell}
    }{\sqrt{1 - z^2}}
    \frac{k!}{(k - \ell)!}
    \left( \frac{u y}{1 + t_1} \right)^{\ell},
\\
    B_1(u, y, \vt)
    &=
    \exp \left(
        - \frac{\log(1 + B_0(u, y, \vt)) - k (k-1) \frac{u^2 t_2 y}{(1 + t_1)^2}}
        {y^2}
    + \frac{(k-1)^2}{4 (1 + t_1)^4}
    + (2 k^2 u^2 - k + 1) \frac{k-1}{4}
    \right),
\\
    B_2(y, \vt)
    &=
    B_1 \left( - \frac{1}{\sqrt{k}}, y, \vt \right)
    + B_1 \left( \frac{1}{\sqrt{k}}, y, \vt \right),
\\
    \phi(t)
    &=
    \frac{t^2}{2} + t - \log(1 + t).
\end{align*}
Then the number of $k$-regular graphs on $n$ vertices
is equal to
\begin{equation}
\label{eq:sgknint}
    \sgkn
    =
    \frac{(n k / e)^{n k / 2} \sqrt{k}}{k!^{n - 1/2}}
    e^{- (k^2 - 1) / 4}
    \left( \frac{n}{2 \pi} \right)^{k/2}
    \int_{\reals_{> - 1} \times \reals^{k-1}}
    B_2(i n^{-1/2}, \vt)
    e^{- n k \phi(t_1) - \sum_{j=2}^k n j t_j^2 / 2}
    d \vt.
\end{equation}
\end{lemma*}

\begin{proof}
\cref{th:had_to_int} is applied to \cref{th:sgkn}
to obtain an integral representation
\[
    \sg^{(k)}_n
    =
    \frac{(-1)^{n k / 2} \sqrt{k!}}{(2 \pi)^{k/2}}
    \int_{\reals^k}
    \left(
        [y^k]
        \frac{e^{-i \sum_{j=1}^k t_j y^j}}
            {\sqrt{1 - y^2}}
    \right)^n
    e^{- \sum_{j=1}^k j t_j^2 / 2}
    d \vt.
\]
In order to make an integral of the form
\[
    \int_{\reals} t^n e^{-t^2/2} dt,
\]
amenable to the Laplace method,
one would apply the change of variable $t \mapsto \sqrt{n} t$
(a similar idea is presented in the Stirling expansion
at the end of \cref{sec:laplace}).
Thus, we apply the variable change $t_j \mapsto \sqrt{n} t_j$
for each $j \in [k]$
\[
    \sg^{(k)}_n
    =
    (-1)^{n k / 2} \sqrt{k!}
    \left( \frac{n}{2 \pi} \right)^{k/2}
    \int_{\reals^k}
    \left(
        [y^k]
        \frac{e^{-i \sum_{j=1}^k \sqrt{n} t_j y^j}}
            {\sqrt{1 - y^2}}
    \right)^n
    e^{- \sum_{j=1}^k n j t_j^2 / 2}
    d \vt.
\]
We would like to expand with respect to negative powers of $n$,
so that each term is negligible
compared to the previous one.
To do so, we apply the identity
$[y^k] F(\alpha y) = \alpha^k [y^k] F(y)$
with $\alpha = \sqrt{n}$
\begin{equation}
\label{eq:sgkn:middle}
    \sg^{(k)}_n
    =
    (-n)^{n k / 2}
    \sqrt{k!}
    \left( \frac{n}{2 \pi} \right)^{k/2}
    \int_{\reals^k}
    \left(
        [y^k]
        \frac{e^{-i \sum_{j=1}^k n^{(1-j)/2} t_j y^j}}
            {\sqrt{1 - y^2 / n}}
    \right)^n
    e^{- \sum_{j=1}^k n j t_j^2 / 2}
    d \vt.
\end{equation}
We develop the central term
as a series in $n^{-1/2}$ rather than a series in $y$
\begin{align}
    [y^k]
    \frac{e^{- i \sum_{j=1}^k n^{(1-j)/2} t_j y^j}}
        {\sqrt{1 - y^2 / n}}
    &=
    \sum_{\ell \geq 0}
    [z^{\ell} y^k]
    \frac{e^{- i \sum_{j=1}^k z^{j-1} t_j y^j}}
        {\sqrt{1 - z^2 y^2}}
    n^{-\ell/2}
\nonumber
    \\&=
    \sum_{\ell = 0}^k
    [z^{\ell} y^{k - \ell}]
    \frac{e^{- i y \sum_{j=1}^k z^{j-1} t_j}}
        {\sqrt{1 - z^2}}
    n^{-\ell/2}
\nonumber
    \\&=
    \sum_{\ell = 0}^k
    [z^{\ell}]
    \frac{
    \left(
        - i
        \sum_{j=1}^k z^{j-1} t_j
    \right)^{k - \ell}}
    {\sqrt{1 - z^2}}
    \frac{n^{-\ell/2}}{(k - \ell)!}
\label{eq:yk}
    \\&=
    \frac{(- i t_1)^k}{k!}
    +
    \sum_{\ell = 1}^k
    [z^{\ell}]
    \frac{
    \left(
        - i
        \sum_{j=1}^k z^{j-1} t_j
    \right)^{k - \ell}}
    {\sqrt{1 - z^2}}
    \frac{n^{-\ell/2}}{(k - \ell)!}
\nonumber
    \\&=
    \frac{(- i t_1)^k}{k!}
    \Bigg(
    1 +
    \sum_{\ell = 1}^k
    [z^{\ell}]
    \frac{
    \left(
        t_1
        + \sum_{j=2}^k z^{j-1} t_j
    \right)^{k - \ell}}
    {\sqrt{1 - z^2}}
    \frac{k!}{(k - \ell)!}
    (-i)^{-\ell}
    t_1^{-k}
    n^{-\ell/2}
    \Bigg)
\nonumber
    \\&=
    \frac{(i t_1)^k}{k!}
    \Bigg(
    1
    + \sum_{\ell = 1}^k
    [z^{\ell}]
    \frac{
    \left(
        1
        + t_1^{-1}
        \sum_{j=2}^k z^{j-1} t_j
    \right)^{k - \ell}}
    {\sqrt{1 - z^2}}
    \frac{k!}{(k - \ell)!}
    (- i t_1)^{-\ell}
    n^{-\ell/2}
    \Bigg).
\nonumber
\end{align}
Define the Laurent polynomial
\[
    A_0(y, \vt) =
    \sum_{\ell = 1}^k
    [z^{\ell}]
    \frac{
    \left(
        1
        + t_1^{-1}
        \sum_{j=2}^k z^{j-1} t_j
    \right)^{k - \ell}}
    {\sqrt{1 - z^2}}
    \frac{k!}{(k - \ell)!}
    (- i t_1)^{-\ell}
    y^{\ell},
\]
where the variable $y$ now represents $n^{-1/2}$.
Then
\[
    \sg^{(k)}_n
    =
    \frac{n^{n k / 2}}{k!^{n - 1/2}}
    \left( \frac{n}{2 \pi} \right)^{k/2}
    \int_{\reals^k}
    t_1^{n k}
    e^{n \log(1 + A_0(n^{-1/2}, \vt))}
    e^{- \sum_{j=1}^k n j t_j^2 / 2}
    d \vt.
\]
Observe that the singularity at $t_1 = 0$
is only apparent
We now isolate the terms of the exponential
that are not converging to $0$ with $n$.
We find
\[
    \log(1 + A_0(y, \vt))
    =
    2 i \binom{k}{2} t_1^{-2} t_2 y
    + \bigO(y^2)
\]
and define
\[
    A_1(y, \vt)
    =
    \exp \left(
    \frac{
        \log(1 + A_0(y, \vt))
        - 2 i \binom{k}{2} t_1^{-2} t_2 y
    }{y^2}
    \right)
\]
so
\[
    \sg^{(k)}_n
    =
    \frac{n^{n k / 2}}{k!^{n - 1/2}}
    \left( \frac{n}{2 \pi} \right)^{k/2}
    \int_{\reals^k}
    t_1^{n k}
    A_1(n^{-1/2}, \vt)
    e^{2 i \binom{k}{2} t_1^{-2} t_2 \sqrt{n}}
    e^{- \sum_{j=1}^k n j t_j^2 / 2}
    d \vt.
\]
To eliminate the $\sqrt{n}$ term in the exponential,
we factorize the quadratic polynomial
\[
    2 i \binom{k}{2} t_1^{-2} t_2 \sqrt{n}
    - n t_2^2 
    =
    - n
    \left( t_2 - i \binom{k}{2} t_1^{-2} n^{-1/2} \right)^2
    - \binom{k}{2}^2 t_1^{-4}.
\]
This motivates the change of variable
$t_2 \mapsto t_2 + i \binom{k}{2} t_1^{-2} n^{-1/2}$
and the definition
\[
    A_2(y, \vt)
    =
    A_1 \left(
        y,
        t_1,
        t_2 + i \binom{k}{2} t_1^{-2} y,
        t_3, \ldots, t_k
    \right)
    e^{- \binom{k}{2}^2 t_1^{-4}}
\]
so
\[
    \sg^{(k)}_n
    =
    \frac{n^{n k / 2}}{k!^{n - 1/2}}
    \left( \frac{n}{2 \pi} \right)^{k/2}
    \int_{\reals^k}
    t_1^{n k}
    A_2(n^{-1/2}, \vt)
    e^{- \sum_{j=1}^k n j t_j^2 / 2}
    d \vt.
\]
Recall that $n k$ is even.
The function $t_1 \mapsto t_1^{n k} e^{- n t_1^2 / 2}$
has two maxima on $\reals$,
reached at $- \sqrt{k}$ and $\sqrt{k}$.
We would rather work with only one,
so we transform our integral.
We cut the integral with respect to $t_1$
into the sum of an integral over $\reals_{< 0}$
and an integral over $\reals_{\geq 0}$.
In the first integral,
the change of variable
$t_1 \mapsto - t_1$ is applied,
and the two integrals are combined.
Set
\[
    A_3(y,\vt) =
    A_2(y, - t_1, t_2, \ldots, t_k)
    + A_2(y, \vt),
\]
then
\[
    \sg^{(k)}_n
    =
    \frac{n^{n k / 2}}{k!^{n - 1/2}}
    \left( \frac{n}{2 \pi} \right)^{k/2}
    \int_{\reals_{> 0} \times \reals^{k-1}}
    A_3(n^{-1/2}, \vt)
    t_1^{n k}
    e^{- \sum_{j=1}^k n j t_j^2 / 2}
    d \vt.
\]
Now we apply the change of variable
$t_1 \mapsto \sqrt{k} (1 + t_1)$
to move the maximum of $t_1^{n k} e^{- n t_1^2 / 2}$
to $0$.
Defining
\[
    \phi(t) = \frac{t^2}{2} + t - \log(1 + t)
\]
and
\[
    A_4(y, \vt) =
    A_3(y, \sqrt{k} (1 + t_1), t_2, \ldots, t_k),
\]
we obtain
\[
    \sg^{(k)}_n
    =
    \frac{(n k / e)^{n k / 2} \sqrt{k}}{k!^{n - 1/2}}
    \left( \frac{n}{2 \pi} \right)^{k/2}
    \int_{\reals_{> - 1} \times \reals^{k-1}}
    A_4(n^{-1/2}, \vt)
    e^{- n k \phi(t_1) - \sum_{j=2}^k n j t_j^2 / 2}
    d \vt.
\]

When we look back at the transformations
from $A_0(y,\vt)$ to $A_4(y, \vt)$,
we find that defining
\begin{align*}
    B_0(u, y, \vt)
    &=
    \sum_{\ell = 1}^k
    [z^{\ell}]
    \frac{
        \left(
            1
            + \frac{u}{1 + t_1}
            \left(
                \frac{k-1}{2} \frac{y z}{(1 + t_1)^2}
                + \sum_{j = 2}^k t_j z^{j-1}
            \right)
        \right)^{k-\ell}
    }{\sqrt{1 - z^2}}
    \frac{k!}{(k - \ell)!}
    \left( \frac{u y}{1 + t_1} \right)^{\ell},
\\
    B_1(u, y, \vt)
    &=
    \exp \left(
        \frac{\frac{u^2 t_2 y}{(1 + t_1)^2} - \log(1 + B_0(u, y, \vt)) - k (k-1)}
        {y^2}
    + \frac{(k-1)^2}{4 (1 + t_1)^4}
    + (2 k^2 u^2 - k + 1) \frac{k-1}{4}
    \right),
\\
    B_2(y, \vt)
    &=
    B_1 \left( - \frac{1}{\sqrt{k}}, y, \vt \right)
    + B_1 \left( \frac{1}{\sqrt{k}}, y, \vt \right)
\end{align*}
ensures
\[
    A_4(y, \vt) = e^{- (k^2 - 1) / 4} B_2(i y, \vt).
\]
This is an improvement, in the sense that
$B_0(u, y, \vt)$, $B_1(u, y, \vt)$ and $B_2(y, \vt)$
are power series with rational coefficients,
while the $A_j(y, \vt)$ were Laurent series
with complex coefficients.
In particular, the argument of the exponential in $B_1(u, y, \vt)$
has valuation at least $1$ in $y$.
\end{proof}

The following technical lemma will be useful
to prove \cref{th:bound:tails}.

\begin{lemma}
\label{th:inttwo}
For any nonnegative integer $d$ and $x > \sqrt{d}$,
we have
\[
    \int_x^{+\infty}
    t^{n d}
    e^{-n t^2/2}
    dt
    \leq
    \frac{x^{n d} e^{- n x^2/2}}
    {n \left( x - \frac{d}{x} \right)}.
\]
Furthermore, for any neighborhood $U$ of $\sqrt{d}$,
this integral is exponentially small compared to
\[
    \int_U
    |t|^{n d}
    e^{-n t^2/2}
    dt.
\]
\end{lemma}

\begin{proof}
Introduce $\phi(t) = \frac{t^2}{2} - d \log(t)$, so
\[
    \int_x^{+\infty}
    t^{n d}
    e^{-n t^2/2}
    dt
    =
    \int_x^{+\infty}
    e^{- n \phi(t)}
    dt.
\]
The derivatives of the function $\phi(t)$ are
\[
    \phi'(t) = t - \frac{d}{t},
    \qquad
    \phi''(t) = 1 + \frac{d}{t^2}.
\]
Thus, $\phi(t)$ is convex
and it stays above its tangeant at point $x$,
so for any $t$
\[
    \phi(t) \geq \phi'(x) (t-x) + \phi(x).
\]
We inject this bound in the integral
\[
    \int_x^{+\infty}
    t^{n d}
    e^{-n t^2/2}
    dt
    \leq
    \int_x^{+\infty}
    e^{- n (\phi'(x) (t-x) + \phi(x))}
    dt.
\]
Since $x > \sqrt{d}$, we have $\phi'(x) > 0$.
We apply the change of variable $y = n \phi'(x) (t - x)$
\[
    \int_x^{+\infty}
    t^{n d}
    e^{-n t^2/2}
    dt
    \leq
    \frac{e^{- n \phi(x)}}{n \phi'(x)}
    \int_0^{+\infty}
    e^{- y}
    d y
    \leq
    \frac{x^{n d} e^{- n x^2/2}}
    {n \left( x - \frac{d}{x} \right)}.
\]

Let us now turn to the second point of the lemma.
For $d = 0$, we have
\[
    \int_U e^{- n t^2/2} dt =
    \exactbigO(n^{-1/2}),
\]
which is exponentially greater than our bound for $d=0$ and $x > 0$
\[
    \frac{x^{n d} e^{- n x^2/2}}
    {n \left( x - \frac{d}{x} \right)}
    =
    \frac{e^{- n x^2/2}}{n x}.
\]
Assume now $d > 0$.
By application of the Laplace method, we have
\[
    \int_U
    |t|^{n d}
    e^{-n t^2/2}
    dt
    =
    \exactbigO
    \left(
    \frac{(d/e)^{nd/2}}{\sqrt{n}}
    \right).
\]
Looking at its derivative,
the function $t \mapsto t^d e^{-t^2/2}$
is strictly decreasing on $\reals_{\geq \sqrt{d}}$,
so
\[
     (d/e)^{d/2} > x^d e^{-x^2/2},
\]
which implies that $\frac{(d/e)^{nd/2}}{\sqrt{n}}$
is exponentially greater
than $\frac{x^{n d} e^{- n x^2/2}} {n \left( x - \frac{d}{x} \right)}$,
concluding the proof.
\end{proof}

\begin{lemma*}[\ref{th:bound:tails}]
For any $k \geq 2$
and small enough neighbordhood $V \subset \reals^k$ of the origin,
there exists $\delta > 0$
such that
the number $\sgkn$ of $k$-regular graphs on $n$ vertices
is equal to
\[
    \sgkn
    =
    \frac{(n k / e)^{n k / 2} \sqrt{k}}{k!^{n - 1/2}}
    e^{- (k^2 - 1) / 4}
    \left(
    \frac{n}{2 \pi} \right)^{k/2}
    \int_V
    B_2(i n^{-1/2}, \vt)
    e^{- n k \phi(t_1) - \sum_{j=2}^k n j t_j^2 / 2}
    d \vt
    \left( 1 + \bigO(e^{-\delta n})\right),
\]
where $B_2(y,\vt)$ is defined in \cref{th:sgkn:integral}.
\end{lemma*}

\begin{proof}
Instead of working directly with the result of \cref{th:sgkn:integral},
we take a few steps back to find an easier expression to handle.
In the proof of \cref{th:sgkn:integral},
let us inject \cref{eq:yk}
into \cref{eq:sgkn:middle}
\[
    \sgkn =
    (-n)^{n k / 2}
    \sqrt{k!}
    \left( \frac{n}{2 \pi} \right)^{k/2}
    \int_{\reals^k}
    \left(
        \sum_{\ell = 0}^k
        [z^{\ell}]
        \frac{
        \left(
            - i
            \sum_{j=1}^k z^{j-1} t_j
        \right)^{k - \ell}}
        {\sqrt{1 - z^2}}
        \frac{n^{-\ell/2}}{(k - \ell)!}
    \right)^n
    e^{- \sum_{j=1}^k n j t_j^2 / 2}
    d \vt.
\]
Following the proof of \cref{th:sgkn:integral},
we observe that in this integral,
a set $U \subset \reals^k$
neighborhood of both $(\sqrt{k}, 0, \ldots, 0)$
and $(-\sqrt{k}, 0, \ldots, 0)$
corresponds, in \cref{eq:sgknint},
to a neighborhood $V \subset \reals^k$ of the origin.
Given $J \subset \reals^k$, let us define
\[
    I_n(J) :=
    n^{(n+1)k/2}
    \int_J
    \left(
    \sum_{\ell=0}^k
    [z^{\ell}]
    \frac{ \left(-i \sum_{j=1}^k t_j z^{j-1} \right)^{k-\ell}}{\sqrt{1-z^2}}
    \frac{n^{-\ell/2}}{(k-\ell)!}
    \right)^n
    e^{-\sum_{j=1}^k n j t_j^2/2}
    d \vt
\]
so $\sgkn = \exactbigO(I_n(\reals^k))$.
We will prove in \cref{th:main:result}
\[
    I_n(U) = \exactbigO \left( \frac{(nk/e)^{nk/2}}{k!^n} \right).
\]
Our goal is thus to prove that
$|I_n(\reals^k \setminus U)|$ is exponentially small
compared to $\frac{(nk/e)^{nk/2}}{k!^n}$.

In the expression of $I_n(J)$,
we isolate the summands corresponding to $\ell = 0$ and $\ell = 1$
and introduce a polynomial $P_k(\vt)$ with nonnegative coefficients
such that
\[
    |I_n(J)| \leq
    \bigO(1)
    n^{(n+1)k/2}
    \int_{J'}
    \left(
    \frac{t_1^k}{k!}
    + \frac{t_1^{k-2}}{(k-2)!}
    t_2 n^{-1/2}
    + \frac{P_k(\vt)}{k!} n^{-1}
    \right)^n
    e^{- \sum_{j=1}^k n j t_j^2/2}
    d \vt.
\]
After factorization by $k!$, and bounding $k (k-1) \leq k^2$,
we obtain
\begin{equation}
\label{eq:IJ}
    |I_n(J)| \leq
    \bigO(1)
    \frac{n^{(n+1)k/2}}{k!^n}
    \int_{J'}
    \left(
    t_1^k
    + k^2 t_1^{k-2} t_2 n^{-1/2}
    + P_k(\vt) n^{-1}
    \right)^n
    e^{- \sum_{j=1}^k n j t_j^2/2}
    d \vt.
\end{equation}
For any nonnegative values $a$, $b$, $c$, we have
\[
    (a + b + c)^n \leq 3^n a^n + 3^n b^n + 3^n c^n
\]
(by considering whether $a$, $b$, or $c$ is the largest),
so we also have the cruder bound
\begin{equation}
\label{eq:IJJ}
    |I_n(J)| \leq
    \bigO(1)
    \frac{n^{(n+1)k/2}}{k!^n}
    \int_{J'}
    \left(
    3^n t_1^{nk}
    + 3^n k^{2n} t_1^{m(k-2)} t_2^n n^{-n/2}
    + 3^n P_k(\vt) n^{-n}
    \right)
    e^{- \sum_{j=1}^k n j t_j^2/2}
    d \vt.
\end{equation}
In the following, we consider successively three cases
\begin{itemize}
\item
$J_1 \subset [- \epsilon, \epsilon]$
for some small enough $\epsilon > 0$,
\item
$J_j \subset \reals_{\geq x}$
for some $j$ and some large enough $x$,
\item
$J \subset \reals \setminus U$ is bounded
and $J_1 \subset \reals \setminus [- \epsilon, \epsilon]$.
\end{itemize}
In each case, we prove that $|I_n(J)|$
is exponentially small compared to $\frac{(nk/e)^{nk/2}}{k!^n}$.
Since this covers all possibilities
for $J \subset \reals \setminus U$,
this will conclude the proof.

\paragraph{Case where $J_1 \subset [- \epsilon, \epsilon]$.}
Then \cref{eq:IJJ} implies
\[
    |I_n(J)| \leq
    \bigO(1)
    \frac{n^{(n+1)k/2}}{k!^n}
    \int_{J'}
    \left(
    3^n \epsilon^{nk}
    + 3^n k^{2n} \epsilon^{m(k-2)} t_2^n n^{-n/2}
    + 3^n P_k(\epsilon, t_2, \ldots, t_k) n^{-n}
    \right)
    e^{- \sum_{j=1}^k n j t_j^2/2}
    d \vt.
\]
and our bound on $|I_n(J)|$ becomes a sum of three terms.
The first term is
\[
    \bigO(1)
    \frac{n^{(n+1)k/2}}{k!^n}
    \int_{\reals_{\geq 0}^{k-1}}
    3^n \epsilon^{nk}
    e^{- n \sum_{j=2}^k t_j^2/2}
    d t_2 \cdots d t_k
    =
    \bigO \left( \frac{n^{nk/2}}{k!^n} (3 \epsilon^k)^n \right).
\]
Choosing $\epsilon > 0$ small enough so that
$3 \epsilon^k \leq (k/e)^{k/2}$
ensures that this bound is exponentially small
compared to $\frac{(nk/e)^{nk/2}}{k!^n}$.
For the second and third terms, we will use the crude bound
\[
    \int_{\reals_{\geq 0}}
    t^{n d}
    e^{-n t^2/2}
    dt
    = \bigO(1)^n,
\]
where the big $\bigO$ depends only
on the nonnegative integer $d$,
obtained for example by application of the Laplace method.
The second term is
\[
    \bigO(1)
    \frac{n^{(n+1)k/2}}{k!^n}
    \int_{\reals_{\geq 0}^{k-1}}
    3^n k^{2n} \epsilon^{n(k-2)} t_2^n n^{-n/2}
    e^{- n \sum_{j=2}^k t_j^2/2}
    d t_2 \cdots d t_k
    =
    \bigO(1)^n \frac{n^{nk/2}}{n^{n/2}},
\]
which is more than exponentially small
compared to $\frac{(nk/e)^{nk/2}}{k!^n}$.
The third term is treated in the same way
\[
    \bigO(1)
    \frac{n^{(n+1)k/2}}{k!^n}
    \int_{\reals_{\geq 0}^{k-1}}
    3^n P_k(\epsilon, t_2, \ldots, t_k)^n n^{-n}
    e^{- n \sum_{j=2}^k t_j^2/2}
    d t_2 \cdots d t_k
    =
    \bigO(1)^n \frac{n^{nk/2}}{n^n}.
\]

\paragraph{Case where $J_j \subset \reals_{\geq x}$ for some $j$ and large enough $x$.}
We apply the same idea as in the previous case.
The bound from \cref{eq:IJJ} is now the sum of the following three terms.
The first term is
\[
    \bigO(1)
    \frac{n^{(n+1)k/2}}{k!^n}
    \int_{J'}
    3^n t_1^{nk}
    e^{- \sum_{j=1}^k n j t_j^2/2}
    d \vt.
\]
Applying our crude bound on the integrals and \cref{th:inttwo},
there exists a constant $C$, depending only on $k$, such that
the first term is asymptotically bounded by
\[
    C^n
    n^{nk/2}
    e^{-n x^2/2}.
\]
We now choose $x$ large enough to ensure
that this is exponentially small
compared to $\frac{(nk/e)^{nk/2}}{k!^n}$.
Namely, we choose $x$ such that
\[
    e^{-x^2/2} < \frac{(k/e)^{k/2}}{C k!}.
\]
The two other terms
\[
    \bigO(1)
    \frac{n^{(n+1)k/2}}{k!^n}
    \int_{J'}
    3^n k^{2n} t_1^{m(k-2)} t_2^n n^{-n/2}
    e^{- \sum_{j=1}^k n j t_j^2/2}
    d \vt
    = \bigO(1)^n \frac{n^{nk/2}}{n^{n/2}}
\]
and
\[
    \bigO(1)
    \frac{n^{(n+1)k/2}}{k!^n}
    \int_{J'}
    3^n P_k(\vt) n^{-n}
    e^{- \sum_{j=1}^k n j t_j^2/2}
    d \vt
    = \bigO(1)^n \frac{n^{nk/2}}{n^n}
\]
and treated as in the previous case.

\paragraph{Case where $J \subset \reals \setminus U$ is bounded and $J_1 \subset \reals \setminus [- \epsilon, \epsilon]$.}
Now $J'_1 \subset \reals_{\geq \epsilon}$,
so we can divide by $t_1$ more easily.
We factorize $t_1$ in our bound on $|I_n(J)|$ from \cref{eq:IJ}
and bound $t_1^{-1}$ by $\epsilon^{-1}$
\[
    |I_n(J)| \leq
    \bigO(1)
    \frac{n^{(n+1)k/2}}{k!^n}
    \int_{J'}
    \left(
    1
    + k^2 \epsilon_1^{-2} t_2 (2n)^{-1/2}
    + \epsilon_1^{-k} P_k(\vt) n^{-1}
    \right)^n
    t_1^{nk}
    e^{- \sum_{j=1}^k n j t_j^2/2}
    d \vt.
\]
We use the inequality $(1+x)^n \leq e^{n x}$
and isolate the summands corresponding to $j=1$ and $j=2$
\[
    |I_n(J)| \leq
    \bigO(1)
    \frac{n^{(n+1)k/2}}{k!^n}
    \int_{J'}
    e^{k^2 \epsilon_1^{-2} t_2 \sqrt{n/2}
        + \epsilon_1^{-k} P_k(\vt)}
    t_1^{nk}
    e^{- n t_1^2/2}
    e^{- n t_2^2/2}
    e^{- n \sum_{j=3}^k t_j^2/2}
    d \vt.
\]
Since $J$ is bounded, the term $e^{\epsilon_1^{-k} P_k(\vt)}$
is bounded.
We factorize
\[
    k^2 \epsilon_1^{-2} t_2 \sqrt{n}
    - n t_2^2/2
    =
    - \frac{n}{2}
    \left( t_2 - k^2 \epsilon_1^{-2} n^{-1/2} \right)
    + \frac{k^4 \epsilon^{-4}}{4}
\]
and shift $t_2$ by $\frac{k^2 \epsilon_1^{-2}}{\sqrt{n}}$,
which is smaller than $1$ for $n$ large enough
\[
    |I_n(J)| \leq
    \bigO(1)
    \frac{n^{(n+1)k/2}}{k!^n}
    \int_{\reals_{\geq 0} \times \reals_{\geq -1} \times \reals_{\geq 0}^{k-2} \setminus U}
    t_1^{nk}
    e^{- \sum_{j=1}^k n j t_j^2/2}
    d \vt.
\]
Because the integration domain stays outside
of the vicinity of $(\sqrt{k}, 0, \ldots, 0)$,
\cref{th:inttwo} implies that it is exponentially small
compared to the full integral
\[
    \frac{n^{(n+1)k/2}}{k!^n}
    \int_{\reals^k}
    t_1^{nk}
    e^{- \sum_{j=1}^k n j t_j^2/2}
    d \vt,
\]
whose asymptotics $\exactbigO \left( \frac{(nk/e)^{nk/2}}{k!^n} \right)$
is extracted by the Laplace method.
\end{proof}

\begin{theorem*}[\ref{th:main:result}]
Assume $k \geq 2$ and $n k$ is even,
then the asymptotic expansion
of the number of $k$-regular graphs
on $n$ vertices is
\[
    \sgkn
    \approx
    \frac{(n k / e)^{n k/2}}{k!^n}
    \frac{e^{- (k^2 - 1) / 4}}{\sqrt{2}}
    \tsgk(n^{-1})
\]
where for all $r$, the $r$th coefficient
of the formal power series $\tsgk(z)$
is a polynomial with rational coefficients
in $k$, $k^{-1}$ and $\indic_{j \leq k}$
for $j \in [3, 2r + 2]$,
explicitly computable using the formula
(where we set $\indic_{2 \leq k} = 1$)
\begin{align*}
    \psi(t)
    &=
    \Bigg(
        1
        + \frac{
            \log \big( \frac{1}{1+t} \big)
            + t
            - \frac{t^2}{2}}
        {t^2}
    \Bigg)^{-1/2},
\\
    T(x)
    &=
    x \psi(T(x)),
\\
    u_{p,q}
    &=
    [s^p] \frac{1}{(1 + T(s))^q},
\\
    v_{p,q}(\vt)
    &=
    [z^p]
    \frac{\left( \sum_{j \geq 2} t_j z^{j-1} \right)^q}
    {\sqrt{1 - z^2}},
\\
    B_{0,r}(u, \vt)
    &=
    \sum_{\substack{1 \leq \ell,\ 0 \leq a \leq \ell,\ 0 \leq b\\ a + b + \ell \leq r}}
    \bigg( \prod_{m=0}^{a + b + \ell - 1} (k-m) \bigg)
    \frac{u^{a + b + \ell} t_1^{r - a - b - \ell}}{a! b!}
    \left( \frac{k-1}{2} \right)^a
    u_{r - a - b - \ell, 3 a + b + \ell}
    v_{\ell - a, b}(\vt),
\\
    C_1(u, s, \vt)
    &=
    \exp \left(
        - \frac{\log \left(
            1
            + \sum_{j \geq 1}
            B_{0,j}(u, \vt)
            s^j \right)
            - k (k-1) \frac{u^2 s^2 t_2}{(1 + T(s t_1))^2}}
        {s^2}
    + \frac{(k-1)^2}{4 (1 + T(s t_1))^4}
    + (2 k^2 u^2 - k + 1) \frac{k-1}{4}
    \right),
\\
    C_2(s, \vt)
    &=
    \left(
    C_1 \left( - \frac{1}{\sqrt{k}}, s, \vt \right)
    + C_1 \left( \frac{1}{\sqrt{k}}, s, \vt \right)
    \right)
    T'(s t_1),
\\
    [z^r] \tsgk(z)
    &=
    (-1)^r
    e^{- t_1^2 / (4k) - \sum_{j=2}^{2r + 2} t_j^2 / (2j)}
    \odot_{\vt = \vone}
    [s^{2r}]
    C_2(s, \vt).
\end{align*}
\end{theorem*}

\begin{proof}
We start with the expression from \cref{th:bound:tails}
\[
    \sgkn
    =
    \frac{(n k / e)^{n k / 2} \sqrt{k}}{k!^{n - 1/2}}
    e^{- (k^2 - 1) / 4}
    \left(
    \frac{n}{2 \pi} \right)^{k/2}
    \int_V
    B_2(i n^{-1/2}, \vt)
    e^{- n k \phi(t_1) - \sum_{j=2}^k n j t_j^2 / 2}
    d \vt
    \left( 1 + \bigO(e^{-\delta n})\right),
\]
for some neighbordhood $V \subset \reals^k$ of the origin,
chosen small enough to ensure $B_2(y, \vt)$
is analytic on $[0,\epsilon] \times V$ for some $\epsilon > 0$.
Let us define
\[
    \psi(t) =
    \left( \frac{\phi(t) - \phi(0)}{\phi''(0) t^2/2} \right)^{-1/2}
    =
    \frac{1}{\sqrt{\phi(t) / t^2}}
    \quad \text{and} \quad
    T(x) = x \psi(T(x))
\]
and the formal power series in $\sqrt{z}$
\[
    \tsg(z) =
    e^{z x_1^2 / (4 k) + \sum_{j=2}^k z x_j^2 / (2 j)}
    \odot_{\vx = \vone}
    B_2 \left( i \sqrt{z}, T(x_1), x_2, \ldots, x_k \right)
    T'(x_1),
\]
the Laplace method from \cref{th:laplace} then implies
\[
    \sgkn
    \approx
    \frac{(n k / e)^{n k/2}}{k!^n \sqrt{2}}
    e^{- (k^2 - 1) / 4}
    \tsg(n^{-1}).
\]
Since $\sgkn$ is real,
any monomial of $B_2(i \sqrt{z}, T(x_1), x_2, \ldots, x_k)$
containing an odd power in $i \sqrt{z}$
must vanish after the Hadamard product
(otherwise, there would be a non real term
in the asymptotic expansion).
Thus, $\tsg(z)$ is in fact a formal power series in $z$.

However, we have not yet reached our goal
of expressing the coefficients of the asymptotic expansion
as functions of a formal value $k$
(in particular, we have sums indexed by $k$ in our expressions).
Consider the $r$th coefficient of $\tsg(z)$
\begin{align}
    [z^r] \tsg(z)
    &=
    [z^r]
    e^{z t_1^2 / (4k) + \sum_{j=2}^k z t_j^2 / (2j)}
    \odot_{\vt = \vone}
    B_2 \left( i \sqrt{z}, T(t_1), t_2, \ldots, t_k \right)
    T'(t_1)
    \nonumber
    \\&=
    [s^{2r}]
    e^{s^2 t_1^2 / (4k) + \sum_{j=2}^k s^2 t_j^2 / (2j)}
    \odot_{\vt = \vone}
    B_2 \left( i s, T(t_1), t_2, \ldots, t_k \right)
    T'(t_1)
    \nonumber
    \\&=
    e^{t_1^2 / (4k) + \sum_{j=2}^k t_j^2 / (2j)}
    \odot_{\vt = \vone}
    [s^{2r}]
    B_2 \left( i s, T(s t_1), s t_2, \ldots, s t_k \right)
    T'(s t_1)
    \nonumber
    \\&=
    e^{t_1^2 / (4k) + \sum_{j=2}^k t_j^2 / (2j)}
    \odot_{\vt = \vone}
    i^{2r}
    [s^{2r}]
    B_2 \left( s, T(- i s t_1), - i s t_2, \ldots, - i s t_k \right)
    T'(- i s t_1)
    \nonumber
    \\&=
    (-1)^r
    e^{- t_1^2 / (4k) - \sum_{j=2}^k t_j^2 / (2j)}
    \odot_{\vt = \vone}
    [s^{2r}]
    B_2 \left( s, T(s t_1), s t_2, \ldots, s t_k \right)
    T'(s t_1)
    \label{eq:FFF}
\end{align}
This motivates the definition of
\begin{align*}
    B_{0,r}(u, \vt)
    &=
    [s^r]
    B_0 \left(u, s, T(s t_1), s t_2, \ldots, s t_k \right)
    \\&=
    [s^r]
    \sum_{\ell=1}^k
    [z^{\ell}]
    \frac{ \left(
        1
        + 
        \frac{k-1}{2}
        \frac{u s z}{(1 + T(s t_1))^3}
        + \frac{u}{1 + T(s t_1)}
        \sum_{j=2}^k s t_j z^{j-1}
    \right)^{k - \ell}}
    {\sqrt{1 - z^2}}
    \frac{k!}{(k-\ell)!}
    \left( \frac{u s}{1 + T(s t_1)} \right)^{\ell}
    \\&=
    \sum_{\ell=1}^k
    \sum_{a + b + \ell \leq k}
    \binom{k - \ell}{a, b, k - \ell - a - b}
    \frac{k! u^{\ell}}{(k-\ell)!}
    [z^{\ell} s^{r - \ell}]
    \frac{
        \left(
            \frac{k-1}{2}
            \frac{u s z}{(1 + T(s t_1))^3}
        \right)^a
        \left(
            \frac{u}{1 + T(s t_1)}
            \sum_{j=2}^k s t_j z^{j-1}
        \right)^b}
    {\sqrt{1 - z^2} (1 + T(s t_1))^{\ell}}
    \\&=
    \sum_{\ell=1}^k
    \sum_{a + b + \ell \leq k}
    \frac{k! u^{\ell + a + b}}
    {a! b! (k - \ell - a - b)!}
    [z^{\ell - a} s^{r - \ell - a - b}]
    \frac{
        \left(
            \frac{k-1}{2}
        \right)^a
        \left(
            \sum_{j=2}^k t_j z^{j-1}
        \right)^b}
    {\sqrt{1 - z^2} (1 + T(s t_1))^{\ell + 3 a + b}}
    \\&=
    \sum_{\ell=1}^k
    \sum_{a + b + \ell \leq k}
    \frac{k! u^{\ell - a - b} t_1^{r - \ell - a - b}}
    {a! b! (k - \ell - a - b)!}
    \left( \frac{k-1}{2} \right)^a
    [z^{\ell - a} s^{r - \ell - a - b}]
    \frac{
        \left(
            \sum_{j=2}^k t_j z^{j-1}
        \right)^b}
    {\sqrt{1 - z^2} (1 + T(s))^{\ell + 3 a + b}}
    \\&=
    \sum_{\substack{1 \leq \ell,\ 0 \leq a \leq \ell,\ 0 \leq b\\ a + b + \ell \leq \min(k, r)}}
    \frac{k! u^{\ell + a + b} t_1^{r - \ell - a - b}}
    {a! b! (k - \ell - a - b)!}
    \left( \frac{k-1}{2} \right)^a
    u_{r - \ell - a - b, \ell + 3 a + b}
    v_{\ell - a, b}(\vt)
    \\&=
    \sum_{\substack{1 \leq \ell,\ 0 \leq a \leq \ell,\ 0 \leq b\\ a + b + \ell \leq r}}
    \frac{k! u^{a + b + \ell} t_1^{r - a - b - \ell}}
    {a! b! (k - a - b - \ell)!}
    \left( \frac{k-1}{2} \right)^a
    u_{r - a - b - \ell, 3 a + b + \ell}
    v_{\ell - a, b}(\vt)
    \indic_{a + b + \ell \leq k}.
\end{align*}
The quotient of factorials, rewritten as a product,
simplifies with the indicator function,
because when $a+b+\ell < k$, the product vanishes
\[
    \frac{k!}{(k - a - b - \ell)!}
    \indic_{a + b + \ell \leq k}
    =
    \bigg(
    \prod_{m=0}^{a + b + \ell - 1}
    (k-m)
    \bigg)
    \indic_{a + b + \ell \leq k}
    =
    \prod_{m=0}^{a + b + \ell - 1}
    (k-m).
\]
Observe that $B_{0,r}(u,\vt)$
is a polynomial with rational coefficients
in the variables $k$,
$\indic_{j \leq k}$ for $j$ in $[1, r]$,
$u$ and $t_1, \ldots, t_r$.
By construction, we have
\[
    B_0(u, s, T(s t_1), s t_2, \ldots, s t_k)
    =
    \sum_{j \geq 0}
    B_{0,j}(u, \vt)
    s^j.
\]
Let us define
\[
    C_1(u, s, \vt) =
    B_1(u, s, T(s t_1), s t_2, \ldots, s t_k).
\]
Then $[s^r] C_1(u, s, \vt)$
is equal to
\[
    [s^r]
    \exp \left(
        - \frac{\log \left(
            1
            + \sum_{j=1}^{r+2}
            B_{0,j}(u, \vt)
            s^j \right)
            - k (k-1) \frac{u^2 s^2 t_2}{(1 + T(s t_1))^2}}
        {s^2}
    + \frac{(k-1)^2}{4 (1 + T(s t_1))^4}
    + (2 k^2 u^2 - k + 1) \frac{k-1}{4}
    \right),
\]
which is a polynomial with rational coefficients
in the variables
$k$,
$\indic_{j \leq k}$ for $j$ in $[1, r + 2]$,
$u$ and $t_1, \ldots, t_{r + 2}$.
Finally, set
\[
    C_2(s, \vt)
    =
    B_2(s, T(s t_1), s t_2, \ldots, s t_k)
    T'(s t_1),
\]
then
\[
    C_2(s, \vt)
    =
    \left(
    C_1 \left( - \frac{1}{\sqrt{k}}, s, \vt \right)
    + C_1 \left( \frac{1}{\sqrt{k}}, s, \vt \right)
    \right)
    T'(s t_1)
\]
and $[s^{2r}] C_2(s, \vt)$
is a polynomial with rational coefficients
in the variables
$k$, $k^{-1}$,
$\indic_{j \leq k}$ for $j$ in $[1, 2r + 2]$,
and $t_1, \ldots, t_{2r + 2}$.

For any positive integer values $k$ and $r$,
$C_2(s,\vt)$ contains no variable $t_j$
for any $j > \min(k, 2r + 2)$
-- or, said otherwise, for all $j$,
any monomial containing $t_j$
also contains $\indic_{\ell \leq k}$
for some $j \leq \ell$.
Thus, we rewrite \cref{eq:FFF} as
\[
    [z^r] \tsg(z)
    =
    (-1)^r
    e^{- t_1^2 / (4k) - \sum_{j=2}^{2r + 2} t_j^2 / (2j)}
    \odot_{\vt = \vone}
    [s^{2r}]
    C_2(s, \vt),
\]
which concludes the proof.
\end{proof}

    \subsection{Proofs of \cref{sec:connected:regular:graphs}}

    \subsubsection{Proofs of \cref{sec:divergent:series}}

\begin{lemma*}[\ref{th:alpha:beta:gamma}]
Consider positive $\alpha$, $\beta$, a real value $\gamma$
and a positive sequence $(a_n)_{n > 0}$ satisfying
$a_n = \exactbigO(n^{\alpha} \beta^n n^{\gamma})$,
then for any fixed $R \in \integers_{> 0}$,
as $n$ tends to infinity, we have
\[
    \sum_{j=R}^{n-R} a_j a_{n-j} =
    \bigO(a_{n-R}).
\]
\end{lemma*}

\begin{proof}
For any fixed $j$, as $n$ tends to infinity, we have
\begin{equation}
\label{eq:anj}
    a_{n-j}
    =
    \exactbigO \left(
        (n-j)^{\alpha (n-j)}
        \beta^{n-j}
        (n-j)^{\gamma}
    \right)
    =
    \exactbigO \left(
        n^{\alpha n}
        \beta^n
        n^{\gamma}
        n^{- \alpha j}
    \right)
    =
    \exactbigO(a_n n^{-\alpha j}).
\end{equation}
Define the function
\[
    f(x) = x^{\alpha x} \beta^x x^{\gamma},
\]
then
\begin{align*}
    \log(f(x)) &=
    \alpha x \log(x) + x \log(\beta) + \gamma \log(x),
\\
    \partial \log(f(x)) &=
    \alpha \log(x) + \alpha + \log(\beta) + \gamma x^{-1},
\\
    \partial^2 \log(f(x)) &=
    \alpha x^{-1} - \gamma x^{-2}.
\end{align*}
The second derivative is positive for $x > \frac{\gamma}{\alpha}$.
Thus $f(j)$ is log-convex on $j \geq \frac{\gamma}{\alpha}$.
By symmetry, so is $j \mapsto f(n-j)$ on $j \leq n - \frac{\gamma}{\alpha}$
and by product,
$j \mapsto f(j) f(n-j)$ is log-convex
on $j \in \left[ \frac{\gamma}{\alpha}, n - \frac{\gamma}{\alpha} \right]$
and its minimum is reached at $j = n/2$.
Set $J = \lceil \max \left( \frac{\gamma}{\alpha}, R + \frac{1}{\alpha} \right) \rceil$.
\cref{eq:anj} then ensures
\[
    f(n-J)
    =
    \bigO(f(n) n^{- \alpha (R + 1 / \alpha)})
    =
    \bigO(a_{n-R} n^{-1})
\]
and, by log-convexity, for any $j \in [J, n - J]$,
\begin{equation}
\label{eq:ajnj}
    f(j) f(n-j)
    \leq f(J) f(n - J)
    = \bigO(a_{n-R} n^{-1}).
\end{equation}
Since $a_n = \exactbigO(f(n))$,
there exists $C > 0$ such that $a_j \leq C f(j)$ for all $j$.
Let us decompose the sum as follows
\[
    \sum_{j = R}^{n-R}
    a_j a_{n-j}
    \leq
    2 \sum_{j = R}^{J - 1}
    a_j a_{n-j}
    + C^2
    \sum_{j = J}^{n - J}
    f(j) f(n-j).
\]
By \cref{eq:anj},
each of the terms of the first sum is a $\bigO(a_{n-R})$
and \cref{eq:ajnj} implies that the second sum
is bounded by $n \bigO(a_{n-R} n^{-1}) = \bigO(a_{n-R})$,
which concludes the proof.
\end{proof}

\begin{lemma*}[\ref{th:an_j}]
Consider $\alpha \in \integers_{> 0}$,
$\beta \in \reals_{> 0}$
and $\gamma \in \reals$,
and a sequence $(a_n)_n$
with asymptotic expansion
\[
    a_n \approx
    n^{\alpha n}
    \beta^n
    n^{\gamma}
    \tA(n^{-1})
\]
for some nonzero formal power series $\tA(z)$.
Define the formal power series
\[
    \tA_j(z) =
    e^{- \alpha j}
    \beta^{-j}
    z^{\alpha j}
    (1 - j z)^{\gamma - \alpha j}
    e^{\alpha z^{-1} (\log(1 - j z) + j z)}
    \tA \left( \frac{z}{1 - j z} \right).
\]
Then for any fixed $j \in \integers_{\geq 0}$,
we have as $n$ tends to infinity
\[
    a_{n-j} \approx
    n^{\alpha n}
    \beta^n
    n^{\gamma}
    \tA_j(n^{-1}).
\]
\end{lemma*}

\begin{proof}
\begin{align*}
    a_{n-j}
    &\approx
    (n-j)^{\alpha (n-j)}
    \beta^{n-j}
    (n-j)^{\gamma}
    \tA((n-j)^{-1})
\\
    &\approx
    n^{\alpha n}
    \beta^n
    n^{\gamma}
    n^{- \alpha j}
    \left( 1 - \frac{j}{n} \right)^{\alpha (n-j) + \gamma}
    \beta^{-j}
    \tA \left( \frac{1}{n} \frac{1}{1 - j/n} \right)
\\
    &\approx
    n^{\alpha n}
    \beta^n
    n^{\gamma}
    \tA_j(n^{-1}).
\end{align*}
\end{proof}

\begin{proposition*}[\ref{th:HA}]
Consider a function $H(z)$ analytic at $0$
and a formal power series
\[
    A(z) = \sum_{n > 0} a_n z^n
\]
whose coefficients satisfy
\[
    a_n \approx
    n^{\alpha n}
    \beta^n
    n^{\gamma}
    \tA(n^{-1})
\]
for some $\alpha \in \integers_{> 0}$,
$\beta \in \reals_{> 0}$,
$\gamma \in \reals$
and nonzero formal power series $\tA(z)$,
then
\[
    [z^n] H(A(z)) \approx
    n^{\alpha n}
    \beta^n
    n^{\gamma}
    \tA_H(n^{-1})
\]
where the formal power series $\tA_H(z)$
is defined as
\begin{align*}
    \tA_H(z) &=
    \sum_{j \geq 0}
    \tA_j(z)
    [x^j] H'(A(x)),
\\
    \tA_j(z) &=
    e^{- \alpha j}
    \beta^{-j}
    z^{\alpha j}
    (1 - j z)^{\gamma - \alpha j}
    e^{\alpha z^{-1} (\log(1 - j z) + j z)}
    \tA \left( \frac{z}{1 - j z} \right).
\end{align*}
\end{proposition*}

\begin{proof}
Direct injection of \cref{th:alpha:beta:gamma}
and \cref{th:an_j} into
\cref{th:divergent:series}.
\end{proof}

\begin{corollary*}[\ref{th:divergent:even}]
Consider a function $H(z)$ analytic at $0$
and a formal power series
\[
    B(z) = \sum_{n > 0} b_n z^n
\]
whose coefficients satisfy $b_n = 0$
for all odd $n$, and for $n$ even
\[
    b_n \approx
    n^{\alpha n}
    \beta^n
    n^{\gamma}
    \tB(n^{-1})
\]
for some $\alpha \in \frac{1}{2} \integers_{> 0}$,
$\beta \in \reals_{> 0}$,
$\gamma \in \reals$
and nonzero formal power series $\tB(z)$.
Then for $n$ even
\[
    [z^n] H(B(z)) \approx
    n^{\alpha n}
    \beta^n
    n^{\gamma}
    \tB_H(n^{-1})
\]
where the formal power series $\tB_H(z)$
is defined as
\begin{align*}
    \tB_H(z) &=
    \sum_{j \geq 0}
    \tB_{2j}(z)
    [x^{2j}] H'(B(x)),
\\
    \tB_j(z) &=
    e^{- \alpha j}
    \beta^{-j}
    z^{\alpha j}
    (1 - j z)^{\gamma - \alpha j}
    e^{\alpha z^{-1} (\log(1 - j z) + j z)}
    \tB \left( \frac{z}{1 - j z} \right).
\end{align*}
\end{corollary*}

\begin{proof}
Set $a_n = b_{2n}$, then
\[
    a_n \approx
    n^{2 \alpha n}
    (2^{2 \alpha} \beta^2)^n
    n^{\gamma}
    2^{\gamma}
    \tB \left( \frac{n^{-1}}{2} \right).
\]
The associated generating function is
\[
    A(z) = \sum_{n > 0} a_n z^n = B(\sqrt{z}).
\]
Observe that $2 \alpha \in \integers_{> 0}$
and $H(B(z))$ is a formal power series in $z^2$:
its only nonzero coefficients have even indices.
Applying \cref{th:HA} yields
\[
    [z^n] H(A(z)) \approx
    n^{2 \alpha n}
    (2^{2 \alpha} \beta^2)^n
    n^{\gamma}
    2^{\gamma}
    \tA_H(n^{-1})
\]
where
\begin{align*}
    \tA_H(z) &=
    \sum_{j \geq 0}
    \tA_j(z)
    [x^j] H'(A(x)),
\\
    \tA_j(z) &= 
    e^{- 2 \alpha j}
    (2^{2 \alpha} \beta^2)^{-j}
    z^{2 \alpha j}
    (1 - j z)^{\gamma - 2 \alpha j}
    e^{2 \alpha z^{-1} (\log(1 - j z) + j z)}
    \tB \left( \frac{1}{2} \frac{z}{1 - j z} \right).
\end{align*}
We define
\[
    \tB_j(z) =
    e^{- \alpha j}
    \beta^{-j}
    z^{\alpha j}
    (1 - j z)^{\gamma - \alpha j}
    e^{2 \alpha z^{-1} (\log(1 - j z) + j z)}
    \tB \left( \frac{z}{1 - j z} \right)
\]
and observe
\[
    \tB_{2j}(z/2) = \tA_j(z).
\]
For $n$ even
\[
    [z^n] H(B(z)) =
    [z^{n/2}] H(A(z))
    \approx
    n^{\alpha n}
    \beta^n
    n^{\gamma}
    \tA_H(2 n^{-1}).
\]
Finally, we replace $\tA_H(2z)$ with
\[
    \tA_H(2z)
    =
    \sum_{j \geq 0}
    \tA_j(2z)
    [x^j] H'(A(x))
    =
    \sum_{j \geq 0}
    \tB_{2j}(z)
    [x^{2j}] H'(B(x)).
\]
\end{proof}

    \subsubsection{Proofs of \cref{sec:asympt:csgkn}}

\begin{theorem*}[\ref{th:csgkn}]
For any $k \geq 3$,
the number $\csgkn$ of connected $k$-regular graphs
on $n$ vertices has asymptotic expansion
\[
    \csgkn \approx
    \frac{(nk/e)^{nk/2}}{k!^n}
    \frac{e^{-(k^2-1)/4}}{\sqrt{2}}
    \tcsgk(n^{-1})
\]
where the formal power series $\tcsgk(z)$
is computed using the following equations
\begin{align*}
    \psi(t) &=
    \left( \frac{t - \log(1+t)}{t^2/2} \right)^{-1/2},
\\
    T(x) &= x \psi(T(x)),
\\
    \tS(z) &= e^{z x^2/2} \odot_{x=1} T'(x),
\\
    f_{k,j}(z) &=
    \begin{cases}
        1 & \text{if } j = 0,
        \\
        \sum_{\ell \geq 3}
        \indic_{k = \ell}
        \indic_{j \, \ell \text{ even}}
        z^{(\ell / 2 - 1) j}
        & \text{if } j > 0,
    \end{cases}
\\
    \tA^{(k)}(z) &= \frac{\tsgk(z)}{\tS(z)},
\\
    \tA^{(k)}_j(z) &=
    \left( \frac{k!}{k^{k/2}} \right)^j
    f_{k,j}(z)
    (1 - j z)^{-1/2 - (k/2 - 1) j}
    e^{(k/2 - 1) z^{-1} (\log(1 - j z) + j z)}
    \tA^{(k)} \left( \frac{z}{1 - j z} \right),
\\
    \tcsgk(z) &=
    \tS(z)
    \sum_{j \geq 0}
    \tA^{(k)}_j(z)
    [x^j]
    \frac{1}{\sgk(x)}.
\end{align*}
\end{theorem*}

\begin{proof}
Define
\[
    H(z) = \log(1 + z),
    \qquad
    A^{(k)}(z) = \sg^{(k)}(z) - 1.
\]
Then, as we saw at the beginning of this section,
\[
    \csg^{(k)}(z) = H(A^{(k)}(z)).
\]
When $k$ is even, for all sufficiently large $n$,
we have $\sgkn > 0$,
so the $n$th coefficients of $A^{(k)}(z)$ is nonzero,
as required by \cref{th:divergent:series}.
If $k$ is odd, $\sgkn = 0$ for all odd $n$
and $\sgkn > 0$ for all large enough even $n$,
as required by \cref{th:divergent:even}.
We have for all $n > 0$
\[
    a^{(k)}_n := [z^n] A^{(k)}(z)
    =
    [z^n] (\sg^{(k)}(z) - 1)
    =
    \frac{\sgkn}{n!}.
\]
We recalled Stirling asymptotic expansion
at the end of \cref{sec:laplace}.
Defining the formal power series
\[
    \psi(t) = \left( \frac{t - \log(1 + t)}{t^2/2} \right)^{-1/2},
    \quad
    T(x) = x \psi(T(x)),
    \quad \text{and} \quad
    \tS(z) =
    e^{z x^2/2} \odot_{x=1} T'(x),
\]
we have
\[
    n! \approx
    n^n e^{-n} \sqrt{2 \pi n} \tS(n^{-1}).
\]
Combined with the asymptotic expansion of $\sgkn$ from \cref{th:main:result},
we deduce, for $k$ fixed and $k n$ even,
\begin{align*}
    a^{(k)}_n
    &\approx
    \frac{1}{n^n e^{-n} \sqrt{2 \pi n} \tS(n^{-1})}
    \frac{(n k / e)^{n k / 2}}{k!^n}
    \frac{e^{-(k^2-1)/4}}{\sqrt{2}}
    \tsgk(n^{-1})
\\
    &\approx
    n^{n (k/2 - 1)}
    \left(
        \frac{e (k/e)^{k/2}}{k!}
    \right)^n
    n^{-1/2}
    \frac{e^{-(k^2-1)/4}}{2 \sqrt{\pi}}
    \frac{\tsgk(n^{-1})}{\tS(n^{-1})}
\end{align*}
We define
\[
    \alpha = \frac{k}{2} - 1,
    \quad
    \beta = \frac{e (k/e)^{k/2}}{k!},
    \quad
    \gamma = -\frac{1}{2},
    \quad
    \tA^{(k)}(z) =
    \frac{\tsgk(z)}{\tS(z)},
\]
so for $k n$ even
\[
    a^{(k)}_n \approx
    \frac{e^{-(k^2-1)/4}}{2 \sqrt{\pi}}
    n^{\alpha n}
    \beta^n
    n^{\gamma}
    \tA^{(k)}(n^{-1}).
\]
The assumption $k \geq 3$ ensures $\alpha > 0$.
For $k$ even, $\alpha$ is a positive integer,
while for $k$ odd, $\alpha \in \frac{1}{2} \integers_{>0}$.
Following \cref{th:HA,th:divergent:even}, we define
\begin{equation}
\label{eq:csg:tAkj}
    \tA^{(k)}_j(z) =
    e^{-\alpha j}
    \beta^{-j}
    \indic_{k j \text{ is even}}
    z^{\alpha j}
    (1 - j z)^{\gamma - \alpha j}
    e^{\alpha z^{-1} (\log(1 - j z) + j z)}
    \tA^{(k)} \left( \frac{z}{1 - j z} \right).
\end{equation}
The factor $\indic_{k j \text{ is even}}$
is always $1$ when $k$ is even
(case of application of \cref{th:HA}),
but is $0$ if $k$ and $j$ are odd
(case of application of \cref{th:divergent:even}).
We obtain
\begin{align*}
    \csg^{(k)}_n
    &=
    n! [z^n] \csg^{(k)}(z)
\\
    &=
    n! [z^n] \log(\sg^{(k)}(z))
\\
    &=
    n! [z^n] H(A^{(k)}(z))
\\
    &\approx
    n!
    \frac{e^{-(k^2-1)/4}}{2 \sqrt{\pi}}
    n^{\alpha n}
    \beta^n
    n^{\gamma}
    \tA^{(k)}_H(n^{-1})
\end{align*}
where the formal power series $\tA^{(k)}_H(z)$ is
\[
    \tA^{(k)}_H(z) =
    \sum_{j \geq 0}
    \tA^{(k)}_j(z)
    [x^j]
    \frac{1}{\sg^{(k)}(x)}.
\]
Defining
\[
    \tcsgk(z) =
    \tS(z) \tA^{(k)}_H(z),
\]
we conclude, after replacing $n!$ with its asymptotic expansion,
\[
    \csg^{(k)}_n
    \approx
    \frac{(n k / e)^{n k / 2}}{k!^n}
    \frac{e^{-(k^2-1)/4}}{\sqrt{2}}
    \tcsgk(n^{-1}).
\]
To ensure the coefficients of the asymptotic expansion
are functions of a formal $k$,
we inject, in the expression of $\tA^{(k)}_j(z)$
from \cref{eq:csg:tAkj},
\[
    \indic_{k j \text{ is even}} z^{\alpha j} =
    \begin{cases}
        1 & \text{if } j = 0,
        \\
        \sum_{\ell \geq 3}
        \indic_{k = \ell}
        \indic_{j\, \ell \text{ even}}
        z^{(\ell / 2 - 1) j}
        & \text{otherwise}
    \end{cases}
\]
for all $j > 0$
(assuming $k \geq 3$ as usual),
and replace $\alpha$, $\beta$ and $\gamma$ with their values,
to obtain the expression of $\tA^{(k)}_j(z)$
stated in the theorem.
\end{proof}

\begin{theorem*}[\ref{th:link:asympt}]
For any fixed $k \geq 3$,
the numbers $\sgkn$ and $\csgkn$
of $k$-regular graphs and connected $k$-regular graphs
on $n$ vertices are linked by the relation
\[
    \csgkn =
    \sgkn
    \left(1 + \exactbigO(n^{- (k+1)(k-2)/2}) \right).
\]
\end{theorem*}

\begin{proof}
A $k$-regular graph contains either $0$ vertex,
or at least $k+1$ vertices
(the equality case corresponds to the complete graph)
so $\sgkn = 0$ for all $n \in [1, k]$.
This implies for all $j \in [1, k]$
\[
    [x^j] \frac{1}{\sgk(x)} = 0
\]
and thus, using the notation $\tcsgk(z)$
from \cref{th:csgkn},
\[
    \tcsgk(z)
    =
    \sum_{j \geq 0}
    \tA^{(k)}_j(z)
    [x^j] \frac{1}{\sgk(x)}
    =
    \tS(z) \tA^{(k)}_0(z) +
    \sum_{j \geq k+1}
    \tA^{(k)}_j(z)
    [x^j] \frac{1}{\sgk(x)}.
\]
We inject
\[
    \tA^{(k)}_0(z)
    =
    \tA^{(k)}(z)
    =
    \frac{\tsg^{(k)}(z)}{\tS(z)}
\]
so
\[
    \tcsgk(z) =
    \tsg^{(k)}(z) +
    \sum_{j \geq k+1}
    \tA^{(k)}_j(z)
    [x^j] \frac{1}{\sgk(x)}.
\]
The \emph{valuation} of a formal power series
is defined as the smallest index corresponding
to a non-zero coefficient.
For example, the valuation of the polynomial $7 z^2 + z^5$
is $2$.
We seek the valuation of
\[
    \tA^{(k)}_j(z) =
    \left( \frac{k!}{k^{k/2}} \right)^j
    f_{k,j}(z)
    (1 - j z)^{-1/2 - (k/2 - 1) j}
    e^{(k/2 - 1) z^{-1} (\log(1 - j z) + j z)}
    \tA^{(k)} \left( \frac{z}{1 - j z} \right)
\]
for $j \geq k+1$.
We have
\[
    \tA^{(k)}(0) = \frac{\tsgk(0)}{\tS(0)} = 2
\]
of valuation $0$,
and
\[
    (1 - j z)^{-1/2 - (k/2 - 1) j}
    e^{(k/2 - 1) z^{-1} (\log(1 - j z) + j z)}
\]
of valuation $0$ as well,
so the valuation of $\tA^{(k)}_j(z)$
is equal to the valuation of
\[
    f_{k,j}(z) =
    \begin{cases}
        1 & \text{if } j = 0,
        \\
        \sum_{\ell \geq 3}
        \indic_{k = \ell}
        \indic_{j\, \ell \text{ even}}
        z^{(\ell / 2 - 1) j}
        & \text{otherwise.}
    \end{cases}
\]
For a fixed $k \geq 3$,
the valuation of $f_{k,j}(z)$ increases with $j$.
At $j = k+1$, we find
\[
    f_{k,j}(z) =
    z^{(k+1)(k-2)/2}.
\]
Thus, $\tA^{(k)}_{k+1}(z)$ has valuation $(k+1)(k-2)/2$
and for any $j > k+1$, the valuation of $\tA^{(k)}_j(z)$
is greater.
This implies
\[
    \tcsgk(z) =
    \tsg^{(k)}(z)
    + \exactbigO(z^{(k+1)(k-2)/2})
    =
    \tsg^{(k)}(z)
    \left(
    1 + \exactbigO(z^{(k+1)(k-2)/2})
    \right).
\]
By \cref{th:main:result,th:csgkn}, we deduce
for
\begin{align*}
    \sgkn
    &\approx
    \frac{(n k / e)^{n k / 2}}{k!^n}
    \frac{e^{-(k^2-1)/4}}{\sqrt{2}}
    \tsgk(n^{-1})
\\
    \csgkn
    &\approx
    \frac{(n k / e)^{n k / 2}}{k!^n}
    \frac{e^{-(k^2-1)/4}}{\sqrt{2}}
    \tcsgk(n^{-1})
\end{align*}
the link
\[
    \csgkn =
    \sgkn
    \left(
    1 + \exactbigO(n^{-(k+1)(k-2)/2})
    \right).
\]
\end{proof}

    \section{Conclusion}

We derived the asymptotic expansion of $k$-regular graphs
and connected $k$-regular graphs.
The expression we obtained for the error terms,
although explicit, are heavy.
We leave as an open problem the search
for simpler expressions
and a combinatorial interpretation,
possibly requiring a change of asymptotic scale.

\paragraph{Acknowledgement.}

I thank Mireille Bousquet-M\'elou and Gilles Schaeffer
for encouraging me to explore the distinction
between an asymptotic expansion computable for any $k$,
and an explicit asymptotic expansion for a formal $k$
(see the discussion just before \cref{sec:def})
during the Alea24 meeting.

Fr\'ed\'eric Chyzak helped me improve the paper
with several constructive remarks,
and notably pointed out that
knowing a linear recurrence with polynomial coefficients
does not automatically translates to an asymptotic expansion.

I also thank Nick Wormald and Brendan McKay
for providing several references I had missed
and expressing interest at a previous publication
\cite{caizergues2023exact}.
Their emails encouraged me to build upon this previous paper
to derive the asymptotic expansion of regular graphs.
Brendan McKay also shared the observation that the expression
of the coefficients of the asymptotic expansion could be simplified,
improving the paper. Finally, he provided many terms of the asymptotic expansion of k-regular graphs that he had guessed, helping ensuring the correctness of the paper.

\bibliographystyle{abbrvnat}
\bibliography{biblio}

\end{document}